\documentclass[10pt]{amsart}
\usepackage{amsmath}
\usepackage{amssymb}
\usepackage{enumerate}
\usepackage{amsbsy}
\usepackage{amsfonts}
\usepackage{amsthm,graphicx,color,yfonts}
\usepackage[latin9]{inputenc}
\usepackage[numbers,sort]{natbib}
\usepackage[colorlinks=true]{hyperref}
\hypersetup{linkcolor=red,citecolor=blue,filecolor=dullmagenta,urlcolor=blue}

\newtheorem{thm}{Theorem}[section]
\newtheorem{lemma}[thm]{Lemma}
\newtheorem{prop}[thm]{Proposition}
\newtheorem{cor}[thm]{Corollary}

\newtheorem{rem}[thm]{Remark}

\makeatother
\numberwithin{equation}{section}
\numberwithin{figure}{section}
\theoremstyle{plain}
\theoremstyle{plain}
\theoremstyle{plain}

\newcommand{\les}{\lesssim}

\newcommand{\vp}{{\varphi}}
\newcommand{\ve}{{\varepsilon}}
\newcommand{\de}{{\delta}}

\newcommand{\al}{{\alpha}}

\newcommand{\R}{{\mathbb R}}
\newcommand{\Z}{{\mathbb Z}}

\newcommand{\rt}{\mathbb{R}^3}
\newcommand{\brad}{\langle D\rangle}
\newcommand{\braxi}{\langle\xi\rangle}
\newcommand{\brat}{\langle t \rangle}

\newcommand{\phase}{{{p}_{\mathbf{\Theta}}}}
\newcommand{\phasep}{ p_{\mathbf{\Theta}}}

\newcommand{\thez}{{\theta_0}}
\newcommand{\theo}{{\theta_1}}
\newcommand{\thet}{{\theta_2}}
\newcommand{\theth}{{\theta_3}}

\newcommand{\thej}{{\theta_j}}

\def\normo#1{\left\|#1\right\|}

\def\abs#1{\left|#1\right|}
\def\bra#1{{\langle#1\rangle}}
\def\bigbra#1{\Big\langle#1\Big\rangle}
\def\wt#1{\widetilde{#1}}
\def\wh#1{\widehat{#1}}

\global\long\def\R{\mathbf{\mathbb{R}}}%
\global\long\def\C{\mathbf{\mathbb{C}}}%
\global\long\def\Z{\mathbf{\mathbb{Z}}}%
\global\long\def\N{\mathbf{\mathbb{N}}}%
\global\long\def\jp#1{\langle#1\rangle}%
\global\long\def\wt#1{\widetilde{#1}}%
\global\long\def\lr#1{\left(#1\right)}%
\global\long\def\ve{\varepsilon}%
\global\long\def\al{\alpha}%
\global\long\def\ep{\epsilon}%

\begin{document}
\title[Modified scattering for Dirac equations]{The modified scattering for Dirac equations of scattering-critical nonlinearity}
\author{Yonggeun Cho}
\email{changocho@jbnu.ac.kr}
\address{Department of Mathematics, and Institute of Pure and Applied Mathematics,
Jeonbuk National University, Jeonju 54896, Republic of Korea}
\author{Soonsik Kwon}
\email{soonsikk@kaist.edu}
\address{Department of Mathematical Sciences, Korea Advanced Institute of Science
and Technology, 291 Daehak-ro, Yuseong-gu, Daejeon 34141, Republic of Korea}
\author{Kiyeon Lee}
\email{kiyeonlee@kaist.ac.kr}
\address{Stochastic Analysis and Application Research Center(SAARC), Korea Advanced Institute of Science and Technology, 291 Daehak-ro, Yuseong-gu, Daejeon, 34141, Republic of	Korea}
\author{Changhun Yang}
\email{chyang@chungbuk.ac.kr}
\address{Department of Mathematics, Chungbuk National University, Chungdae-ro1,
Seowon-gu, Cheongju-si, Chungcheongbuk-do, Republic of Korea}
\begin{abstract}

{ In this paper, we consider the Maxwell-Dirac system in 3 dimension under zero magnetic field. We prove the global well-posedness and modified scattering for small solutions in the weighted Sobolev class. Imposing the Lorenz gauge condition, (and taking the Dirac projection operator), it becomes a system of Dirac equations with Hartree type nonlinearity with a long range potential as $|x|^{-1} $.  We perform the weighted energy estimates. In this procedure, we have to deal with various resonance functions that stem from the Dirac projections. We use the spacetime resonance argument of Germain-Masmoudi-Shatah (\cite{gemasha2008, gemasha2012-jmpa, gemasha2012-annals}), as well as the spinorial null-structure. On the way, we recognize a long range interaction which is responsible for a logarithmic phase correction in the modified scattering statement. }
\end{abstract}

\maketitle

\section{Introduction}
Consider the Maxwell-Dirac equations  \eqref{maxwell-dirac} in $\R\times\R^{3}$:
\begin{align}
\left\{ \begin{aligned}(\partial_{t}+igA_{0}) & {\Psi}=\sum_{j=1}^{3}\al^{j}(\partial_{j}+igA_{j})\Psi-im\beta\Psi\qquad\mathrm{in}\;\;\mathbb{R}^{1+3},\\
-\square A_{0} & =g\bra{\Psi,\Psi},\\
\square A_{j} & =g\bra{\Psi, \al^j\Psi}.
\end{aligned}
\right.\label{maxwell-dirac} \tag{MD}
\end{align}
Here the unknown $\Psi:\R\times\R^{3}\to\mathbb{C}^{4}$
is the spinor field, $A=(A_0,A_1,A_2,A_3)$ is given real-valued $1$-form, and the $4\times4$ matrices $\alpha^{j}$'s and
$\beta$ are Dirac matrices as follows:
\begin{align*}
\al^{j}=\left(\begin{array}{ll}
0 & \sigma^{j}\\
\sigma^{j} & \;0
\end{array}\right),\qquad\beta=\left(\begin{array}{ll}
I_{2} & \;\;\;0\\
0 & -I_{2}
\end{array}\right),
\end{align*}
where $\sigma^{j}$'s are Pauli matrices given by
\begin{align*}
\sigma^{1}=\left(\begin{array}{ll}
0 & \;1\\
1 & \;0
\end{array}\right),\quad\sigma^{2}=\left(\begin{array}{ll}
0 & \;-i\\
i & \;\;\;0
\end{array}\right)\quad\sigma^{3}=\left(\begin{array}{ll}
1 & \;\;\;0\\
0 & \;-1
\end{array}\right).
\end{align*}
The constant $m > 0$ is a mass and $g>0$ is a coupling constant. The
d'Alembert operator $\square$ has the form  $\square:=\partial_{t}^{2}-\Delta$ and the inner product is defined by $\bra{\psi,\phi} = \psi^{\dagger}\phi$.

The equations \eqref{maxwell-dirac} model an electron in electromagnetic field and form a fundamental system in quantum electrodynamics. 
\eqref{maxwell-dirac} is invariant under the gauge transformation: $(\Psi,A)\to(e^{i \chi}\Psi,A-d \chi)$
for a real-valued function $\chi$ on $\R\times\R^{3}$. For concreteness of our discussion, let us choose Lorenz gauge, which is defined by the condition  
\begin{align*}
-\partial_{t}A_{0}+\sum_{j=1}^{3}\partial_{j}A_{j}=0.
\end{align*}
By assuming vanishing magnetic field 
$${\rm \text{curl}(A_{1},A_{2},A_{3})=0}$$
one can find a function $\vp$ such that $\nabla\vp = (A_{1},A_{2},A_{3})$.
The gauge transformation $\psi(t,x):=e^{ig\vp(t,x)}\Psi(t,x)$
uncouples the equations \eqref{maxwell-dirac} to satisfy the Dirac equations \eqref{maineq} 
\begin{align}
\left\{ \begin{aligned}\left(\partial_{t}+\sum_{j=1}^{3}\al^{j}\partial_{j}+im\beta\right)\psi & =ic_1\left(|x|^{-1}*|\psi|^{2}\right)\psi\;\;\mathrm{in}\;\;\mathbb{R}^{1+3},\\
\psi(0) & =\psi_{0},
\end{aligned}
\right.\label{maineq}  \tag{DE}
\end{align}
where  $c_1 = \frac{g^2}{4\pi}$ and $|\psi|^2=\langle \psi,\psi\rangle$. We refer to \cite{chagla} for rigorous derivation. 
Any smooth solution to \eqref{maineq} satisfies the $L_x^2$ conservation laws:
\begin{align}
\|\psi(t)\|_{L_{x}^{2}}=\|\psi_{0}\|_{L_{x}^{2}}.\label{mass-conserve}
\end{align}
By a scaling in $L^2$ we set $m = 1$ hereafter.
The purpose of this paper is to develop a modified scattering theory of global solutions to \eqref{maineq} for small initial data.




As observed in \cite{anfosel}, the square of linear operator is diagonalized
\begin{align}
\lr{\sum_{j=1}^{3}\al_{j}\cdot\xi_{j}+\beta}^{2}=\jp{\xi}^{2}I_{4},\label{diago}
\end{align}
where $\langle\xi\rangle := (1+|\xi|^{2})^{\frac{1}{2}}$
for frequency vector $\xi\in\mathbb{R}^{3}$. 
We introduce the Dirac projection operators $\Pi_{\pm}(D)$
defined by
\begin{align}\label{eq-def-proj}
\Pi_{\pm}(D):=\frac{1}{2}\left(I_4\pm\frac{1}{\brad}\Big[\sum_{j=1}^{3}\al^{j}D_{j}+\beta\Big]\right),
\end{align}
where $D = (D_1, D_2, D_3), D_j = -i\partial_j$, and $\mathcal F(\left<D\right> f) = \langle\xi\rangle \wh{f}$. Here $\mathcal{F}(f)(\xi) = \wh{f}(\xi) = \int_{\mathbb{R}^3} e^{-ix\cdot \xi}f(x)\,dx$.  Then using \eqref{diago},
we get
\[
i\Big( -i\sum_{j=1}^{3}\al^{j}\partial_{j}+\beta \Big)=\jp D(\Pi_{+}(D)-\Pi_{-}(D)).
\]
By definition of projections we further get
\begin{equation}
\Pi_{+}(D)+\Pi_{-}(D)=I_{4},\quad \Pi_{\pm}(D)\Pi_{\pm}(D)=\Pi_{\pm}(D),\quad \Pi_{\pm}(D)\Pi_{\mp}(D)=0. \label{proj-commu}
\end{equation}
By $\psi_{\pm}$ we denote $\Pi_{\pm}(D)\psi$.
Then, the Dirac equation \eqref{maineq} can be rewritten as:
\begin{align}
\left\{ \begin{aligned}
(i\partial_{t} -\brad)\psi_{+} & = -c_1\Pi_{+}(D)\Big[(|x|^{-1}*|\psi|^{2})\psi\Big],\\
(i\partial_{t}+ \brad)\psi_{-} & = - c_1\Pi_{-}(D)\Big[(|x|^{-1}*|\psi|^{2})\psi\Big],\\
\psi_{+}(0)  = \psi_{0,+}:= & \Pi_{+}(D)\psi_{0} 
\text{ and }
\psi_{-}(0)  =\psi_{0,-}:=\Pi_{-}(D)\psi_{0}, 
\end{aligned}
\right.\label{maineq-decou}
\end{align} 
which is a system of half Klein-Gordon equations coupled by Hartree nonlinearity. Denoting by $e^{\mp it\brad}\psi_{0,\pm}$ the
free solutions to \eqref{maineq-decou}
\[
e^{\mp it\brad}\psi_{0,\pm}(x) = \frac{1}{(2\pi)^{3}}\int_{\mathbb{R}^{3}}e^{i(x\cdot\xi\mp t\langle\xi\rangle)}\widehat{\psi_{0,\pm}}(\xi)\,d\xi,
\]
we can express the free solution to \eqref{maineq} as
\[
U(t)\psi_{0}:=e^{-it\brad}\Pi_{+}(D)\psi_{0}+e^{it\brad}\Pi_{-}(D)\psi_{0}.
\]
By Duhamel's principle, the solutions to \eqref{maineq-decou} satisfy the following integral equation
\begin{align}
\left\{ \begin{aligned}
\psi_{+}(t) & =e^{- it\langle D\rangle}\psi_{0,+}+i\int_{0}^{t}e^{- i(t-t')\langle D\rangle}\Pi_{+}(D)\Big[(|x|^{-1}*|\psi|^{2})\psi\Big](t')\,dt', \\ 
\psi_{-}(t) & =e^{ it\langle D\rangle}\psi_{0,-}+i\int_{0}^{t}e^{ i(t-t')\langle D\rangle}\Pi_{-}(D)\Big[(|x|^{-1}*|\psi|^{2})\psi\Big](t')\,dt'.
\end{aligned} \label{inteq0}  \right. 
\end{align}

There has been  considerable mathematical interest in the Cauchy problem for \eqref{maxwell-dirac}. 
The local well-posedness (LWP) for small smooth initial data was first proved by Groos \cite{gro}.
Later, Bournaveas \cite{bour} considered rough initial data and obtained (LWP) result in $H^s(\mathbb{R}^3)$ for $s>\frac12$.
D'Ancona, Foschi and Selberg \cite{anfosel2} proved almost optimal (LWP) in $H^s(\mathbb{R}^3)$ with $s > 0$ in the sense that \eqref{maxwell-dirac} is $L^2$ critical, which means that it leaves the $L^2$-mass invariant by a scaling.
\eqref{maxwell-dirac} on other dimensions are also widely studied. 
Huh \cite{huh} considered \eqref{maxwell-dirac} on $\mathbb{R}^{1+1}$ and proved global well-posedness (GWP) in  $L_{x}^{2}(\R)$ and asymptotic behavior of the solutions.  D'Ancona and Selberg \cite{ansel} considered \eqref{maxwell-dirac} on $\mathbb{R}^{1+2}$ and established (GWP) in $L_x^2(\mathbb{R}^2)$.
Recently, Gavrus and Oh \cite{gaoh} established a linear scattering theory of \eqref{maxwell-dirac} on $\R^{1+d}$ for $d\ge4$ under the Coulomb gauge condition.

The Cauchy problem for \eqref{maineq} also has been extensively studied by several authors. Chadam and Glassey \cite{chagla} showed the existence of a unique global solution for smooth initial data with compact support when $d=2$.
Herr and Tesfahun \cite{Herr2015} proved a small data scattering of \eqref{maineq} with a potential replaced by Yukawa potential $e^{-|x|}|x|^{-1}$ in $H^s(\R^3)$ for $s>\frac12$. As explained in \cite[Remark~1.2]{Herr2015}, (LWP) for \eqref{maineq} in $H^s(\R^3)$ for $s>\frac14$ can be derived by a slight modification of the argument in Herr and Lenzmann \cite{hele2014}.


 In this paper, we are interested in asymptotic behaviors of small global solutions to \eqref{maineq}. The equation is scaling-critical in the sense that the Duhamel terms decay as fast as the linear solutions and so it cannot be seen purely perturbative term. A heuristic computation shows that a long range potential in the nonlinearity gives $|x|^{-1}*|e^{-\theta it \left<D\right>}\psi^\infty|^2 \sim t^{-1}$ if $\psi^\infty \neq 0$ in $L^2(\R^3)$. Indeed, in \cite{chleoz, chhole}\footnote{ In \cite{chleoz,chhole}, the authors proved the  nonexistence of
linear scattering for the equations with $\psi^\dagger \beta \psi$ rather than $|\psi|^2$. But almost the same argument can be
applied to \eqref{maineq-decou}.}, the authors showed that the linear scattering is not possible in $L^{2}$ space. Hence, one can anticipate a modified scattering, which means solutions decay like linear solutions and converge to linear-like 
solutions after \emph{logarithmic phase correction}. Here, we state our main result of global existence and modified scattering.

\begin{thm} \label{mainthm} Let $k\ge1000$. There exists $\overline{\ve_{0}}> 0$ such that:

$(i)$\;\; Suppose that $\psi_0:\R^3\rightarrow\C^4$ is small in weighted spaces as follows:{
\begin{align}
\sum_{\theta \in \{ +,-\} }\left[\|\psi_{0,\theta}\|_{H^{k}}+\|\langle x\rangle^{2}\psi_{0,\theta}\|_{H^{2}}+\|\jp{\xi}^{10}\widehat{\psi_{0,\theta}}\|_{L_{\xi}^{\infty}}\right] < \ve_{0}\label{condition-initial}
\end{align}
for $\ve_{0}$ with $0 < \ve_{0} < \overline{\ve_{0}}$. Then the Cauchy problem \eqref{maineq} with initial data $\psi_{0}$
has a unique global solution $\psi$ to \eqref{maineq}  decaying as}
\begin{align}
\|\psi(t)\|_{L_{x}^{\infty}}\les\ve_{0}\bra{t}^{-\frac{3}{2}}.\label{global-bound}
\end{align}
{
$(ii)$\;\;Moreover, there exists a free solution $U(t)\phi$ such that
\begin{align}
\left\Vert \braxi^{10}\mathcal{F}\Big[\psi(t)-U_{B}(t)U(t)\phi\Big](\xi)\right\Vert _{L_{\xi}^{\infty}}\les\bra{t}^{-\de_{0}}\ve_{0},\label{scattering}
\end{align}
for some $0<\delta_0<\frac{1}{100}$, where the phase corrections are given by
\[
U_{B}(t) = \sum_{\theta \in \{  \pm\} }e^{-iB(t, \theta D)}\Pi_{\theta}(D)
\]
and}
\begin{align}
\left\{ \begin{aligned}B(t,\xi) & =B_{+}(t,\xi)+B_{-}(t,\xi),\\
B_{\theta}(t,\xi) & =\frac{c_1}{(2\pi)^{3}}\int_{0}^{t}\int_{\mathbb{R}^{3}}\left|\frac{\xi}{\langle\xi\rangle}-\theta\frac{\sigma}{\langle\sigma\rangle}\right|^{-1}\abs{\wh{\psi_{\theta}}(\sigma)}^{2}d\sigma\frac{\rho(s^{-\frac{1}{100}}\xi)}{\langle s\rangle}ds,
\end{aligned}
\right.\label{modified-phase}
\end{align}
for $t>0$, a real constant $c_1$ and $\rho \in C_0^\infty(B(0,2))$. \\
A similar result holds for $t<0$ by time reversal symmetry.
\end{thm}
\begin{rem}
In Theorem \ref{mainthm}, the scattering sense \eqref{scattering} is described in the Fourier space. This is because the energy estimates are carried out in $L_\xi^\infty$. (See Propositions \ref{prop-scatt}.) But it can also be expressed in the physical space by the following observation
\[
\left\| \psi(t) - U_B(t)U(t) \phi \right\|_{L_x^2} \les \left\|\bra{\xi}^{10}\mathcal F \left[ \psi(t) - U_B(t)U(t) \phi \right] \,\right\|_{L_\xi^\infty} \xrightarrow{t\to \infty} 0.
\]
\end{rem}
The modified scattering of small solutions often occurs when the nonlinearity contains long-range interactions. This topic has been studied by many authors. Ozawa \cite{Ozawa91} showed the modified scattering for 1D cubic nonlinear Schr\"odinger equations (NLS). Hayahsi and Naumkin \cite{Hayashi-Naumkin} showed it for high dimensional NLS. Later, it turns out that the spacetime-resonance argument by Germain, Masmoudi, and Shatah is efficient for these problems  \cite{kapu,pusa,iopu,gepuro}. Among others, the most relevant reference to this work is the modified scattering result by Pusateri \cite{pusa} for Boson star equation: 
\begin{align}\label{bosoneq}(-i\partial_t + \left<D\right>) u = \left(|x|^{-1}*|u|^2\right)u.\end{align}
 Using the Dirac projection, \eqref{maineq} is written in \eqref{maineq-decou}, which has a similar linear structure to \eqref{bosoneq}. But the scalar equation \eqref{bosoneq} has a single resonance function ${p} = \braxi - \langle\xi-\eta\rangle - \langle\eta +\sigma \rangle + \bra{\sigma}$ in the interaction representation. In particular, since $\nabla_\xi p$ gives a null-structure near $\eta = 0$,  a possible loss of time factor can be recovered when weighted energy estimates\footnote{The weighted energy estimates are referred as the control of quantity $x^2\psi_\theta$ in $H^2$} are concerned.  However, in \eqref{maineq-decou} each nonlinear term contains all signs of Dirac projections, and it gives rise to various phases of interaction,
\[
{p}_{{\mathbf{\Theta}} } = \thez\braxi-\theo\langle\xi-\eta\rangle-\thet\langle\eta+\sigma\rangle+\theth\bra{\sigma},\;\; {\mathbf{\Theta}} =(\thez,\theo,\thet,\theth),
\]
with $\theta_j \in\{+,-\}$ (see \eqref{inteq-f}). Hence the aforementioned null-structures arising from $\nabla_\xi {p}_{{\mathbf{\Theta}}}$ may be lost for some combinations of signs. This causes a significant obstacle in the course of weighted energy estimates.
To overcome this obstacle, for some cases, we utilize the spinorial null-structure occurring from the Dirac projections $\Pi_\theta$, which is successful except for the cases $\theta_0 \neq \theta_1, \theta_2 = \theta_3$. The remaining cases $\thez \neq \theo, \thet = \thet$ can be treated by the usual time non-resonance (integration by parts in time) (see \textbf{case b} in Section \ref{sec-energy}).

The proof is based on the \textit{bootstrap argument}. We construct a function space adapted to time decay estimates below \eqref{timedecay}. By assuming the norm is a priori small, we deduce that the solutions decay like a solution to a linear equation. To achieve this, we first establish the weighted energy estimates, which require to control $x\psi_\theta $ and $x^2 \psi_\theta$ in $H^2$. Here, several null conditions, as well as time non-resonance, play an important role in the course of estimates (see Section~\ref{subidea}). And then, for asymptotic analysis, we approximate the interaction function $f_{\pm} = e^{\pm it\langle D\rangle}\psi_{\pm}$ in the Fourier space by using the Taylor expansion to figure out the leading contribution where the phase logarithm correction is required. For rigorous analysis, we decompose the singular potential in the Fourier space, $\mathcal{F}(|x|^{-1})\sim |\eta|^{-2}$, with a suitable scale in $\eta$ depending on time, say $t^{L_0}$.
Then, $|\eta|\les t^{L_0}$ corresponds to the main term, and, at the same time, remaining contributions on $|\eta|\gtrsim t^{L_0}$ are shown to be integrable $O(t^{-1-})$ via the integration by parts in space variables.


The paper is organized as follows: In Section~2, we introduce a function space adapted to time decay estimates. Under a priori assumption, we establish several frequency localized multilinear estimates. In Section~3, we investigate the null structures which improve the estimates obtained in Section~2.
The improved estimates play a crucial role in the course of the weighted energy estimates.
In Section~4, we provide the proof of the main theorem by assuming the weighted energy estimates and $L_{\xi}^{\infty}$ estimates, followed by the proof of these two estimates in the last two sections, respectively.

\subsection{Notations}
\noindent $\bullet$ (Mixed-normed spaces) For a Banach space $X$
and an interval $I$, $u\in L_{I}^{q}X$ iff $u(t)\in X$ for a.e. $t\in I$
and $\|u\|_{L_{I}^{q}X}:=\|\|u(t)\|_{X}\|_{L_{I}^{q}}<\infty$. Especially,
we denote $L_{I}^{q}L_{x}^{r}=L_{t}^{q}(I;L_{x}^{r}(\rt))$, $L_{I,x}^{q}=L_{I}^{q}L_{x}^{q}$,
$L_{t}^{q}L_{x}^{r}=L_{\mathbb{R}}^{q}L_{x}^{r}$.

\noindent  $\bullet$ For $1\le p<\infty$, $W^{s,p}:=\{ f :  \|f\|_{W^{s,p}}= \| \langle D\rangle^s f \|_{L^p}<\infty \}$. As usual, we denote $W^{s,2}$ by $H^s$, the Sobolev spaces. 
For $k\in \N$, $\|f\|_{W^{k,\infty}}:=\sum_{|\alpha|\le k} \|\partial^\alpha f \|_{L^\infty}$, where $\alpha$ denotes the multi-index $\alpha=(\alpha_1,\alpha_2,\alpha_3)$, $\partial^\alpha=\partial_{1}^{\alpha_1}\partial_{2}^{\alpha_2}\partial_3^{\alpha_{3}}$ and $|\alpha|=\alpha_1+\alpha_2+\alpha_3$.

\noindent $\bullet$ As usual different positive constants depending
only on the coupling constant $g$ are denoted by the same letter $C$, if not specified.
$A\lesssim B$ and $A\gtrsim B$ means that $A\le CB$ and $A\ge C^{-1}B$,
respectively for some $C>0$. $A\sim B$ means that $A\lesssim B$
and $A\gtrsim B$.

\noindent $\bullet$ (Dyadic decomposition) Let $\rho$ be a
bump function such that $\rho\in C_{0}^{\infty}(B(0,2))$
with $\rho(\xi)=1$ for $|\xi|\le1$ and define $\rho_{N}(\xi):=\rho\left(\frac{\xi}{N}\right)-\rho\left(\frac{2\xi}{N}\right)$
for $N\in2^{\mathbb{Z}}$. Then we define the frequency dyadic projection operator 
$P_{N}$ by $\mathcal{F}(P_{N}f)(\xi)=\rho_{N}(\xi)\widehat{f}(\xi)$.
Similarly, we define $P_{\le N}$ and $P_{\sim N}$ as the fourier multipliers with symbols $\rho_{\le N_{0}}:=1-\sum_{N>N_{0}}\rho_{N}$ and $\widetilde{\rho_{N}}:=\rho_{N/2}+\rho_{N}+\rho_{2N}$, respectively. We observe that  $\widetilde{P_{N}}P_{N}=P_{N}\widetilde{P_{N}}=P_{N}$. In addition, we denote 
$P_{N_{1}\le\cdot\le N_{2}}:=\sum_{N_{1}\le N\le N_{2}}P_{N}$.
Especially, we denote $P_{N}f$ simply by $f_{N}$ for any measurable function
$f$.

\noindent $\bullet$ Let $\textbf{A}=(A_i), \textbf{B}=(B_i) \in \R^n$. Then $\textbf{A} \otimes \textbf{B}$ denotes the usual tensor product such that $(\textbf{A} \otimes \textbf{B})_{ij} = A_iB_j$. We also denote a tensor product of $\textbf{A} \in \C^n$ and $\textbf{B} \in \C^m$  by a matrix $\textbf{A} \otimes \textbf{B} = (A_iB_j)_{\substack{i=1,\cdots,n\\i=1,\cdots,m}}$. We often denote $\textbf{A} \otimes \textbf{B}$ simply by  $\textbf{A}\textbf{B}$. We also consider the product of  $x \in \R^3$ and $f \in \C^4$ by $xf = x \otimes f$. Moreover, we denote $\nabla f := (\partial_j f_i)_{ij}$.

\section{Function space and Time decay}

\subsection{Time decay of Dirac equations}For given small initial data as in \eqref{condition-initial}, by
a standard local theory in weighted energy spaces, we have a small
local solution $\psi_{\theta}(t)$ on $[0,T_{0}]$ for $\theta \in \{ +,- \}$. Our goal is to show the solutions are global and decay in time. Let us begin with introducing a
spacetime norm incorporating time decay. 
Let $k\ge300$ and $0<\delta_0<\frac{1}{100}$.
For $\ve_{1}>0$ to be chosen
later, we assume a priori smallness of solutions: for a large time $T>0,$
\begin{align}
\|\psi\|_{\Sigma_{T}}
:=\|\Pi_+(D)\psi\|_{\Sigma_{T}^+} + \|\Pi_-(D)\psi\|_{\Sigma_{T}^-}
\les\ve_{1},\label{assumption-apriori}
\end{align}
where 
\begin{align*}
\begin{aligned}\|\phi\|_{\Sigma_{T}^{\pm}} & :=\sup_{t\in[0,T]}\Big[\brat^{-\de_{0}}\|\phi(t)\|_{H^{k}}+\brat^{-\de_{0}}\|\bra{x} e^{\pm it\jp D}\phi(t)\|_{H^{2}}\phantom{]}\\
 & \phantom{[}\qquad\qquad\qquad\quad+\brat^{-2\de_{0}}\|\langle x\rangle^{2}e^{\pm it\jp D}\phi(t)\|_{H^{2}}+\left\Vert \braxi^{10}\widehat{\phi(t)}\right\Vert _{L_{\xi}^{\infty}}\Big].
\end{aligned}
\end{align*}
One may relax regularity assumptions,
e.g., $H^{k}$, $\braxi^{10}$, or $\delta_0$. We
do not pursue to optimize those parameters. 

In order to obtain pointwise decay of solutions from the above a priori bound, we use the following linear estimates.
\begin{prop}[Linear decay estimates]\label{timedecay-prop}
Let $f:\R^3\rightarrow \C^4$. For any $t\in\R$ one has 
\begin{align}\begin{aligned}\label{timedecay}
\left\| e^{\pm it \langle D \rangle }f \right\|_{L_x^\infty(\R^3)}
&\lesssim \frac{1}{(1+|t|)^\frac32}
\left\| (1+|\xi|)^6\widehat{f}(\xi) \right\|_{L_\xi^\infty(\R^3)}  \\
&\qquad + \frac{1}{(1+|t|)^\frac{31}{20}}
\left\{  \left\| \langle x\rangle^2 f\right\|_{L^2(\R^3)} 
+ \| f \|_{H^{50}(\R^3)}    \right\}.
\end{aligned}\end{align}
\end{prop}
We refer to \cite{pusa} for the proof of \eqref{timedecay}. 
As a consequence of Proposition~\ref{timedecay-prop}
and a priori assumption \eqref{assumption-apriori}, there exists $C$ such that 
\begin{align}\label{W2infty}
	\| \psi(t) \|_{W^{2,\infty}} \le 
	\| \Pi_+(D)\psi(t) \|_{W^{2,\infty}} + \|\Pi_-(D) \psi(t) \|_{W^{2,\infty}}
	\le C\ep_1(1+t)^{-\frac32}
\end{align}
for any $0\le t \le T$.

\subsection{Nonlinear term estimates}
Let us denote the Hartree nonlinear term by 
\begin{align*}
	\mathcal{N}(\psi_1,\psi_2,\psi_3)(s):= \left(|x|^{-1}*\left\langle \psi_3(s),\psi_2(s) \right\rangle \right)\psi_1(s).
\end{align*}
In the next lemma, we introduce several inequalities for the Hartree nonlinear term. The proof is an immediate consequence of the Hardy-Littlewood-Sobolev inequality, so we omit it.
\begin{lemma}[Hartree nonlinear term estimates]\label{lem:Har HK}
Let $\psi_i:\R^{1+3}\rightarrow \C^4$ for $i=1,2,3$.
Then,
\begin{align}
\left\| \mathcal{N}(\psi,\psi,\psi) \right\|_{H^m(\R^3)} &\les  \| \psi\|_{L_x^2}\|\psi\|_{H^m}\|\psi\|_{L_x^6}, \quad \text { for } m\in  \N \cup \{0\}, \nonumber \\ 
\left\| \mathcal{N}(\psi_1,\psi_2,\psi_3) \right\|_{H^2(\R^3)} &\les  \|\psi_1\|_{H^2} \left( \|\psi_2\|_{H^2}\|\psi_3\|_{L_x^6}
+ \|\psi_2\|_{L_x^6}\|\psi_3\|_{H^2}
\right), \label{Har Hk}\\ 
\left\| \mathcal{N}(\psi_1,\psi_2,\psi_3) \right\|_{W^{2,\infty}(\R^3)} &\les 
\|\psi_1\|_{W^{2,\infty}} \left( \|\psi_2\|_{H^{4}}\|\psi_3\|_{L_x^6}
+ \|\psi_2\|_{L_x^6}\|\psi_3\|_{H^{4}}\right). \nonumber
\end{align}
\end{lemma}
As a consequence, we have the following time decay estimates for the nonlinear term when $\psi$ is in $\sum_T$.
\begin{cor}\label{cor:Har}
	Let $\psi$ be solution to \eqref{maineq} satisfying the a priori assumption \eqref{assumption-apriori}. Then, 
	\begin{align}
		\begin{aligned}\label{NHk}
	\left\| \mathcal{N}(\psi,\psi,\psi)(s) \right\|_{H^k(\R^3)} &\les  \langle s\rangle^{-1+\delta_0}\ep_1^3, \\ 
	\left\| \mathcal{N}(\psi,\psi,\psi)(s) \right\|_{W^{2,\infty}(\R^3)} &\les  \langle s\rangle^{-\frac52}\ep_1^3. 
\end{aligned}
	\end{align}
\end{cor}
\begin{proof}
By the previous Lemma, we only suffice to bound $L^6$ norm. 
Interpolating a priori decay assumption \eqref{W2infty} and the conservation law \eqref{mass-conserve}, we have  
\begin{align*}
\|\psi(s)\|_{L_x^6(\R^3)} \les \langle s\rangle^{-1}\ep_1.
\end{align*}
Therefore Lemma \ref{lem:Har HK} finishes the proof of Corollary \ref{cor:Har}.
\end{proof}

Next, we establish frequency localized estimates, which will be used several times in the course of the energy estimates. The proof is quite standard (see \cite{pusa,chhwya}), but we provide the proof for self-containedness.
\begin{lemma}\label{Lem:FL} Let $\psi_i:\R^{1+3}\rightarrow \C^4$ for $i=1,2$\ and $N\in 2^{\Z}$ be a dyadic number. Then, we get
\begin{align*}
	\|P_{N}\langle\psi_{1},\psi_{2}\rangle\|_{L_x^{2}(\R^3)} &\les  N^{\frac{3}{2}}\bra{N}^{-5}\left\Vert \widehat{\langle D\rangle^{10}\psi_1}\right\Vert _{L_x^{\infty}}
	\left\Vert \widehat{\langle D\rangle^{10}\psi_2}\right\Vert _{L_x^{\infty}}, \\ 
	\|P_{N}\langle\psi_{1},\psi_{2}\rangle\|_{L_x^{\infty}(\R^3)} & \les\langle N\rangle^{-2}\|\psi_{1}\|_{W^{2,\infty}}\|\psi_{2}\|_{W^{2,\infty}},    \\ 
	\|P_{N}\langle\psi_{1},\psi_{2}\rangle\|_{L_x^{\infty}(\R^3)}&\les N^{3}\|\psi_{1}\|_{L_x^{2}}\|\psi_{2}\|_{L_x^{2}}.
\end{align*}
\end{lemma}
\begin{proof}
By Plancherel's theorem and H\"older inequality, we see that
\begin{align*}
	&\|P_{N}\langle\psi_{1},\psi_{2}\rangle\|_{L_x^{2}}^{2} 
	 \les\int_{\R^3}\rho_{N}^{2}(\xi)\abs{\mathcal{F}\langle\psi_{1},\psi_{2}\rangle(\xi)}^{2}d\xi\\
	 & \les \int_{\R^3}\rho_{N}^{2}(\xi) \left| \int_{\R^3} \langle\eta\rangle^{-10}\langle\xi-\eta\rangle^{-10} \Big\langle \widehat{\langle D\rangle^{10}\psi_1}(\eta),\widehat{\langle D\rangle^{10}\psi_2}(\xi-\eta) \Big\rangle d\eta \right|^2 d\xi\\ 
	 & \les \left\Vert \widehat{\langle D\rangle^{10}\psi_1}\right\Vert _{L_x^{\infty}}
	 \left\Vert \widehat{\langle D\rangle^{10}\psi_2}\right\Vert _{L_x^{\infty}} \int_{\R^3}\rho_{N}^{2}(\xi) \left|\int_{\R^3}\langle\eta\rangle^{-10}\langle\xi-\eta\rangle^{-10} d\eta \right|^{2}d\xi\\
	 & \les N^{3}\langle N\rangle^{-10}\left\Vert \widehat{\langle D\rangle^{10}\psi_1}\right\Vert _{L_x^{\infty}}
	 \left\Vert \widehat{\langle D\rangle^{10}\psi_2}\right\Vert _{L_x^{\infty}}.
	\end{align*}
Next, we estimate 
	\begin{align}\label{proof-a3}
	\begin{aligned}\|P_{N}\langle\psi_{1},\psi_{2}\rangle\|_{L_x^{\infty}} & =\sup_{x\in\R^3}\left|\int_{\R^3} e^{ix\cdot\xi}\rho_{N}(\xi)\mathcal{F}\langle\psi_{1},\psi_{2}\rangle(\xi) d\xi\right|\\
	 & =\sup_{x\in\R^3}\left|\int_{\R^3} e^{ix\cdot\xi}\rho_{N}(\xi)\langle\xi\rangle^{-2}\mathcal{F}\left[(1-\Delta)\langle\psi_{1},\psi_{2}\rangle\right](\xi)d\xi \right|\\
	 & \les\|\mathcal{F}^{-1}(\rho_{N}\langle\xi\rangle^{-2})*(1-\Delta)\langle\psi_{1},\psi_{2}\rangle\|_{L_{x}^{\infty}}\\
	 & \les\|\mathcal{F}^{-1}(\rho_{N}\langle\xi\rangle^{-2})\|_{L_{x}^{1}}\|(1-\Delta)\langle\psi_{1},\psi_{2}\rangle\|_{L_{x}^{\infty}}\\
	 & \les\langle N\rangle^{-2}\|\psi_{1}\|_{W^{2,\infty}}\|\psi_{2}\|_{W^{2,\infty}}.
	\end{aligned}
\end{align}
We also have
\begin{align*}
		\|P_{N}\langle\psi_{1},\psi_{2}\rangle\|_{L_x^{\infty}}&\les\|\rho_{N}\mathcal{F}\langle\psi_{1},\psi_{2}\rangle\|_{L_x^{1}}\les\|\rho_{N}\|_{L_x^{1}}\|\mathcal{F}\langle\psi_{1},\psi_{2}\rangle\|_{L_x^{\infty}}\\
		&\les N^{3}\|\psi_{1}\|_{L_x^{2}}\|\psi_{2}\|_{L_x^{2}}.
		\end{align*}
\end{proof}
As a direct consequence of the Lemma~\ref{Lem:FL}, we find time decay of localized bilinear term with $\psi$ satisfying the a priori assumption \eqref{assumption-apriori}.
\begin{cor}
	Let $\theta_j\in\{+,-\}$ and 
	$\psi_{j}$ satisfy  $\|\psi_{j}\|_{\Sigma_{T}^{\theta_j}}\les\ve_{1}$ for $j=1,2$. For a dyadic number $N\in 2^{\Z}$, we have 
	\begin{align}
	\|P_{N}\langle\psi_{1},\psi_{2}\rangle(s)\|_{L_{x}^{2}} & \les N^{\frac{3}{2}}\bra{N}^{-5}\ve_{1}^{2}, \label{norm-two} \\ 
	\|P_{N}\langle\psi_{1},\psi_{2}\rangle(s)\|_{L_{x}^{\infty}} & \les\min(\langle N\rangle^{-2}\langle s\rangle^{-3},N^{3})\ve_{1}^{2}.\label{norm-infty}
	\end{align}
\end{cor}
Also, the time decay estimates for the frequency localized Hartree term can be easily obtained from Lemma~\ref{Lem:FL} and Corollary~\ref{cor:Har}.
\begin{cor} Let $\psi$ and $\phi$ satisfy \eqref{assumption-apriori}. For a dyadic number $N\in 2^{\Z}$, we have 
	\begin{align}
	\| P_N \mathcal{N}(\psi,\psi,\psi)(s) \|_{L_x^2} &\les \min\left( N^\frac32\langle s\rangle^{-\frac52}, \langle N\rangle^{-10} \langle s\rangle^{-1}  \right)\ep_1^3,\label{esti-g-2}\\ 
	\left\|P_{N}\big\langle \mathcal{N}(\psi,\psi,\psi)  ,\phi\big\rangle(s)\right\|_{L_{x}^{\infty}}
	&\les\min\left( \langle N\rangle^{-2}\langle s\rangle^{-4}, N^{3}\langle s\rangle^{-1} \right)\ve_{1}^{4}.\label{esti-g-1}
\end{align}
\end{cor}

We close this section by introducing a useful lemma about pseudo-product operators which will repeatedly used in our analysis.
\begin{lemma}[Coifman-Meyer operator estimates]\label{kernel}
Assume that a multiplier $\textbf{m}$ satisfies 
\[
\|\textbf{m}\|_{\rm CM} := \left\Vert \iint_{\mathbb{R}^{3+3}}\mathbf{m}(\xi,\eta)e^{ix\cdot\xi}e^{iy\cdot\eta}d\xi d\eta\right\Vert _{L_{x,y}^{1}(\R^{3}\times\R^3)} < \infty.
\]
Then, for $\frac{1}{p}+\frac{1}{q}=\frac{1}{2}$,
\begin{align*}
\left\Vert \int_{\mathbb{R}^{3}}\mathbf{m}(\xi,\eta)\widehat{\psi}(\xi\pm\eta)\widehat{\phi}(\eta)d\eta\right\Vert _{L_{\xi}^{2}}\les \|\textbf{m}\|_{\rm CM}\|\psi\|_{L^{p}}\|\phi\|_{L^{q}},
\end{align*}
and for $\frac{1}{p}+\frac{1}{q}+\frac{1}{r} =1$,
\begin{align*}
	\left| \iint_{\mathbb{R}^{3+3}}\mathbf{m}(\eta,\sigma)\widehat{\psi_1}(\eta\pm\sigma)\widehat{\psi_2}(\eta)\wh{\psi_3}(\sigma) d\eta  d\sigma \right| \les \|\textbf{m}\|_{\rm CM}\|\psi_1\|_{L^{p}}\|\psi_2\|_{L^{q}}\|\psi_3\|_{L^{r}}.
\end{align*}
\end{lemma}


\section{Null-structure}
To prove Theorem \ref{mainthm}, we have to use structure of the equation obtained from the Dirac projection operator \eqref{eq-def-proj} and the resonance function. Lemmas in this section will be used throughout the proofs of Section \ref{sec-energy} and \ref{sec-infty}.

From \eqref{eq-def-proj}, the symbol of Dirac projection operators are given by 
\begin{align*}
 \Pi_{\pm}(\xi) = \frac12 \left( I \pm \frac{\al \cdot \xi + \beta}{\braxi}\right).
\end{align*}
The symbol are bounded in the sense that 
\begin{align}\label{bound of projection}
	\left| \nabla^n \Pi_{\pm}(\xi) \right| \lesssim \frac{1}{\langle \xi\rangle^{n-1}},
\end{align}
which implies the boundedness of  Dirac projection operators in the Sobolev spaces: For $1<p<\infty$ and $s\ge 0$,
\begin{align*}
	\| \Pi_{\pm}(D)\psi \|_{W^{s,p}} \les \| \psi\|_{W^{s,p}} .
\end{align*}
Furthermore, one can easily verify the null structure  
\begin{align}\label{spinorial}
	\Pi_{\pm}(\xi)\Pi_{\mp}(\xi)=0,
\end{align}
which proved very useful especially when we deal with bilinear operators in the nonlinear term (\cite{anfosel,anfosel2,Bejenaru2015,Bejenaru2017}). 
\begin{lemma}\label{smallness}
For $\xi,\eta\in\mathbb{R}^{3}$ satisfying $|\eta|\ll |\xi|$, we have
	\begin{align}\label{decay-null}
		\Big|\nabla_{\eta}^m\nabla_{\xi}^{n}\left[\Pi_{\pm}(\xi)\Pi_{\mp}(\xi-\eta)\right]\Big|&\les |\eta|^{1-m}\max\left(\frac{1}{\braxi^{n+1}},\frac{1}{\bra{\xi- \eta}^{n+1}}\right).
	\end{align}
\end{lemma} 

\begin{proof}
We only consider $\pm=+$.
We compute  
 	\begin{align*}
	& 4\Pi_{+}(\xi) \Pi_{-}(\xi - \eta) \\
	&= \left( I + \frac{\al \cdot \xi + \beta}{\braxi}\right) \left( I - \frac{\al \cdot (\xi - \eta) + \beta}{\langle \xi - \eta \rangle}\right) \\
	&=  \al \cdot \left( \frac{\xi}{\braxi} - \frac{\xi - \eta}{\langle \xi - \eta \rangle}  \right)  + \beta \left( \frac1\braxi - \frac1{\langle \xi -\eta \rangle}  \right) - \frac{   (\al \cdot \xi)(\al \cdot \eta)  }{\braxi \langle \xi - \eta \rangle}\\
	&\hspace{6cm} + \frac{    (\al \cdot \eta) \beta }{\braxi \langle \xi - \eta \rangle} +  \frac{(\eta -2\xi)\cdot\eta I}{\langle \xi -\eta \rangle  \left( \braxi + \langle \xi -\eta \rangle \right)}.
\end{align*}
For the first term, we apply the mean value theorem
\begin{align*}
\frac{ \xi}{\braxi} - \frac{ \xi - \eta}{\langle \xi - \eta \rangle} = \int_0^1 \frac{1}{\xi_\lambda}
\left(I-\frac{\xi_\lambda\otimes\xi_\lambda}{\langle \xi_\lambda \rangle^2}\right) \eta d\lambda,
\end{align*}
where $\xi_\lambda=\xi+\lambda(\xi-\eta)$. Then, a direct computation gives that 
\begin{align*}
	\left| \nabla_\xi^n \left( \frac{ \xi}{\braxi} - \frac{ \xi - \eta}{\langle \xi - \eta \rangle}  \right) \right| \les  \frac{| \eta|}{\langle \xi  \rangle^{n+1}}.
\end{align*}
The second term can be estimated similarly.  
The estimates for the remaining three terms are easily dealt with thanks to \eqref{bound of projection}.
\end{proof}

Using the Lemma~\ref{smallness}, we get the following bilinear estimates.
\begin{lemma}\label{lem-null-bilinear} Let $N \in 2^{\Z}$ and $\theta \in \{ +,- \}$. Assume that $\|\psi_1\|_{\sum_T^{\theta}}, \|\psi_2\|_{\sum_T^{-\theta}}\les \ep_1$. Then we get
	\begin{align}\label{null-bilinear}
		\|P_{N}\bra{\Pi_{\theta}(D)\psi_{1}(t),\Pi_{-\theta}(D)\psi_{2}(t)}\|_{L_{x}^{\infty}}\les N\bra{N}^{-2}\bra{t}^{-3+ \frac{\de_0}{10}}\ve_{1}^{2}.
	\end{align}
\end{lemma} 
\begin{rem}
As a direct consequence of \eqref{norm-infty} and the boundedness of Dirac projections \eqref{bound of projection}, one has 
\begin{align*}
	\|P_{N}\bra{\Pi_{\theta}(D)\psi_{1}(t),\Pi_{-\theta}(D)\psi_{2}(t)}\|_{L_{x}^{\infty}}\les \bra{N}^{-2}\bra{t}^{-3+ \frac{\de_0}{10}}\ve_{1}^{2}.
\end{align*}
Thus, by exploiting the spinorial null structure in \eqref{null-bilinear}, we obtain an extra $N$ factor which is useful when $N\le 1$.
\end{rem}

\newcommand{\mult}{{ K_{(N,N_1,N_2)}(\xi,\sigma) }}
\begin{proof} 
As in \eqref{proof-a3}, one can obtain  
\begin{align*}
&\|P_{N}\bra{\Pi_{\theta}(D)\psi_{1}(t),\Pi_{-\theta}(D)\psi_{2}(t)}\|_{L_{x}^{\infty}}  \\
&\quad \les A_N \||D|^{-\ve}\bra{D}^{2}\psi_1(t)\|_{L_x^{\infty}}\||D|^{-\ve}\bra{D}^{2}\psi_2(t)\|_{L_x^{\infty}},
\end{align*}
where 
\begin{align*}
&A_N\\
&=\int\!\!\!\! \int_{\R^{3+3}} \left| \int\!\!\!\!\int_{\R^{3+3}} e^{iy\cdot \xi} e^{iz\cdot \sigma} \bra{\xi+\sigma}^{-2} \bra{\sigma}^{-2}|\xi +\sigma|^{\ve}|\sigma|^{\ve} \Pi_{\theta}(\sigma)\Pi_{-\theta}(\xi + \sigma) d\xi d\sigma\right| dz dy,
\end{align*}
for $\ve$ satisfying $0<\left( 2+ \frac23\de_0\right)\ve\ll \frac{\de_0}{10}$.
Since
\begin{align*}
\||D|^{-\ve} \phi \|_{L_x^\infty} \les \|\phi\|_{L_x^2}^{\frac23\ve}\|\phi\|_{L_x^\infty}^{1-\frac23\ve} \text{ for } \phi \in L_x^2 \cap L_x^\infty \text{ and } 0 < \ve \ll 1,	
\end{align*}
we have by \eqref{DepD2} and the a priori assumption \eqref{assumption-apriori} that 
\begin{align}\label{DepD2}
	\||D|^{-\ve}\bra{D}^{2}\psi_j(t)\|_{L_x^{\infty}}
	\les \bra{t}^{-\frac{3}{2} + \ve +\frac13 \ve \de_0}\ep_1, \text{ for } j=1,2. 
\end{align}
Thus, we suffice to prove that 
\begin{align}\label{eq:bound-an}
	A_N \les N\langle N\rangle^{-2}.
\end{align}
We apply the dyadic decomposition to obtain 
\[
A_N\le \sum_{N_1,N_2 \in 2^\Z}A_{(N,N_1,N_2)},
\]
where 
\begin{align*}
A_{(N,N_1,N_2)}=\iint \left| \iint e^{iy\cdot \xi} e^{iz\cdot \sigma} \mult  d\xi d\sigma\right| dz dy,
\end{align*}
with
\begin{align}
\begin{aligned}\label{eq:k-multiplier}
&\mult\\
&\;\; := \rho_{N}(\xi)\rho_{N_1}(\xi + \sigma)\rho_{N_2}(\sigma)\bra{\xi+\sigma}^{-2} \bra{\sigma}^{-2}|\xi +\sigma|^{\ve}|\sigma|^{\ve} \Pi_{\theta}(\sigma)\Pi_{-\theta}(\xi + \sigma).
\end{aligned}
\end{align}
A direct computation using Lemma \ref{smallness} gives that for $N\les N_1$, 
\begin{align*}
\left| \nabla_\sigma^m \nabla_\xi^n \mult \right| \les \langle N_1\rangle^{-2}\langle N_2\rangle^{-2} N N_1^\ep N_2^\ep N^{-n}N_2^{-m}.
\end{align*}
Integrating by parts and using the differential inequalities, we estimate 
\begin{align}\begin{aligned}\label{standard calculation}
	&\sum_{N\les N_1 }A_{(N,N_1,N_2)}  \\ 
	&\les  \sum_{N\les N_1 } \bigg( \normo{|z|^{-2}}_{L^1(|z|\le N_2^{-1})}\int    \normo{\int e^{iy \cdot \xi} \Delta_\sigma \mult d\xi }_{L_\sigma^1} dy\\
	&\hspace{1.5cm} +\normo{|z|^{-4}}_{L^1(|z|\ge N_2^{-1})}\int    \normo{\int e^{iy \cdot \xi} \Delta_\sigma^2 \mult d\xi }_{L_\sigma^1} dy \bigg) \\
	& \les\sum_{N\les N_1 } \bigg(  \normo{|y|^{-2}}_{L^1(|y|\le N^{-1})} \normo{\Delta_\xi (N_2^{-1}\Delta_\sigma +N_2 \Delta_\sigma^2) \mult }_{ L_\sigma^1L_\xi^1} \bigg) \\
	&\hspace{1.5cm}+ \normo{|y|^{-4}}_{L^1(|y|\ge N^{-1})}\normo{\Delta_\xi^2(N_2^{-1}\Delta_\sigma + N_2\Delta_\sigma^2) \mult }_{ L_\sigma^1L_\xi^1}\\
	&\les \sum_{N\les N_1 } N \bra{N_1}^{-2} \bra{N_2}^{-2} N_1^{\ve}N_2^{\ve}  \les N \bra{N}^{-2}.
\end{aligned}\end{align}
For the remaining contribution where $N_1\ll N$, we can  obtain the desired bound via the change of variables by the symmetry of $N_1$ and $N_2$. Indeed, by defining
\[
\wt{K}_{(N,N_1,N_2)}(\xi,\sigma) := K_{(N,N_1,N_2)}(\xi, -\xi + \sigma),
\]
we have
\[\sum_{N_1\ll N}A_{(N,N_1,N_2)}
=\sum_{N_1\ll N}\iint \left| \iint e^{iy'\cdot \xi} e^{iz\cdot \sigma} \wt{K}_{(N,N_1,N_2)}(\xi,\sigma)  d\xi d\sigma\right| dz dy'.
\]
Now, one can verify that $\wt{K}_{(N,N_1,N_2)}$ satisfies the differential inequalities
\begin{align*}
	\left| \nabla_\sigma^m \nabla_\xi^n \wt{K}_{(N,N_1,N_2)}(\xi,\sigma) \right| \les \langle N_1\rangle^{-2}\langle N_2\rangle^{-2} N N_1^\ep N_2^\ep N^{-n}N_1^{-m}.
\end{align*}
Similar computation as in \eqref{standard calculation}
yields that $\sum_{N_1\ll N}A_{(N,N_1,N_2)} A_N \les N \bra{N}^{-2}$. 
\end{proof}

Following lemma concerns another null-structure.
\begin{lemma}\label{lem:linearpart} Let $N \in 2^{\Z}$. Assume that $\psi_1,\psi_2$ satisfy a priori assumption \eqref{assumption-apriori}. Then,
	\begin{align}\label{eq:linearpart}
	\begin{aligned}
	&\left\|P_{N}\left(\bigbra{\bra{D}\psi_1(t),\psi_2(t)} - \bigbra{\psi_1(t),\bra{D}\psi_2(t)} \right) \right\|_{L_{x}^{\infty}(\R^3)}\\
	&\hspace{5cm}\les N\min \left(  \bra{N}^{-2}\bra{t}^{-3+ \frac{\de_0}{10}} , N^3 \right) \ve_{1}^{2}.
    \end{aligned}	
    \end{align}
\end{lemma}

\begin{rem}
A slight modification of the proof of \eqref{norm-infty} gives that 
\begin{align}\label{eq:linearpart2}
	\left\|P_{N} \bra{\bra{D}\psi(t),\psi(t)}  \right\|_{L_{x}^{\infty}(\R^3)} 
	\les \min \left(  \bra{N}^{-1}\bra{t}^{-3+ \frac{\de_0}{10}} , N^3 \right).
\end{align}
Thus, by exploiting the structure in the left side of \eqref{eq:linearpart}, we indeed obtained an extra $N$ gain  when $N\le1$, compared to the one obtained just by applying the triangle inequality and \eqref{eq:linearpart2}.
\end{rem}

\begin{proof} 
The proof of the first bounds proceeds similarly to the proof of Lemma~\ref{lem-null-bilinear}.
We estimate by using \eqref{DepD2}
\begin{align*}
	&\left\|P_{N}\left(\bigbra{\bra{D}\psi_1(t),\psi_2(t)} - \bigbra{\psi_1(t),\bra{D}\psi_2(t)} \right) \right\|_{L_{x}^{\infty}(\R^3)}\\
	&\les B_N \left\||D|^{-\ve}\bra{D}^{2}\psi_{1}(t)\right\|_{L_x^{\infty}} \left\||D|^{-\ve}\bra{D}^{2}\psi_{2}(t)\right\|_{L_x^{\infty}}\\
	&\les  B_N \bra{t}^{-3 + 2\ve +\frac23 \ve \de_0}\ve_1^2,
\end{align*}
where
\begin{align*}
	B_N:=\iint \left| \iint e^{iy\cdot \xi} e^{iz\cdot \sigma} 
	\bra{\xi+\sigma}^{-2} \bra{\sigma}^{-2}|\xi +\sigma|^{\ve}|\sigma|^{\ve} \big( \bra{\sigma} - \bra{\xi + \sigma} \big) 
	d\xi d\sigma\right| dz dy
\end{align*}
for $\ve$ satisfying $0<\left( 2+ \frac23\de_0\right)\ve\ll \frac{\de_0}{10}$.
In the integrand of $B_N$, the cancellation term $\Big(\bra{\sigma} -  \bra{\xi + \sigma}\Big)$ plays a alternative role to the spinorial null-structure $\Pi_{\theta}(\sigma)\Pi_{-\theta}(\xi + \sigma)$ in \eqref{eq:k-multiplier}. Thus, in a similar way to the proof of \eqref{eq:bound-an}, one can obtain
\[
B_N \les N \bra{N}^{-2}.
\]

	Next, we show the bounds in \eqref{eq:linearpart}. By the mean value theorem and H\"older inequality, we estimate 
	\begin{align*}
	&\left\|P_{N}\left(\bigbra{\bra{D}\psi_{1}(t),\psi_{2}(t)} - \bigbra{\psi_{1}(t),\bra{D}\psi_{2}(t)} \right) \right\|_{L_{x}^{\infty}(\R^3)} \\
	&\les \left\| \int \rho_{N}(\xi)  \Big( \bra{\sigma} -  \bra{\xi+\sigma} \Big) \bigbra{\wh{\psi_{1}}(t,\sigma),\wh{\psi_{2}}(t,\xi+\sigma)}   d\sigma  \right\|_{L_\xi^1}\\
	&\les \Big\|\rho_N(\xi) |\xi|\Big\|_{L_\xi^1}  \|\psi_1\|_{L^2}\|\psi_2\|_{L^2}
	\les N^4 \ve_1^2.
\end{align*}
\end{proof}

\section{Proof of Theorem \ref{mainthm}}

For any $\psi_{0}$ satisfying \eqref{condition-initial}, it is quite standard to show the existence of local solution $\psi(t)$ in $\Sigma_{T}$ for some $T$ (for instance see \cite{choz2006-siam, hele2014, chhwya}).
	Then the solution to \eqref{maineq-decou} can be extended globally by a continuity
	argument. To achieve this we have only to  prove the following: 
	Given any $T > 0$, let $\psi$ be a solution with initial data satisfying \eqref{condition-initial} on $[0,T]$. For some small $\varepsilon_1 > 0$ we assume that
	\[
	\|\psi\|_{\Sigma_{T}}\le K\ve_{1}.
	\]
	Then there exists $C$ depending only on $K$ such that
	\begin{align}
		\|\psi\|_{\Sigma_{T}} \le \ve_{0} + C \ve_{1}^{3}.\label{claim}
	\end{align}
	Combining \eqref{claim} and Proposition \ref{timedecay-prop},  we obtain the bound \eqref{global-bound}.
The proof of \eqref{claim} will be done if we prove  Propositions \ref{prop-energy} and \ref{prop-scatt} below. Recall that for simplicity we denote $\Pi_{\theta}(D)\psi$ by $\psi_{\theta}$ for $\theta\in\{+,-\}$, so 
\begin{align*}
	\psi=\psi_+ + \psi_-.
\end{align*}

\begin{prop}[Weighted energy estimate]\label{prop-energy} 
Assume that $\psi\in C([0,T],H^{k})$ satisfies
the condition \eqref{assumption-apriori}.
 Then we obtain the following estimates: For $\theta_0\in\{+,-\}$,
\begin{align}
 & \sup_{t\in[0,T]}\langle t\rangle^{-\de_{0}}\|\psi_{\thez}(t)\|_{H^{s}}\le\ve_{0}+C\ve_{1}^{3},\label{estimate-n}\\
 & \sup_{t\in[0,T]}\langle t\rangle^{-\de_{0}}\|\bra{x} e^{\thez it\brad}\psi_{\thez}(t)\|_{H^{2}}\le\ve_{0}+C\ve_{1}^{3},\label{estimate-1}\\
 & \sup_{t\in[0,T]}\langle t\rangle^{-2\de_{0}}\|\langle x\rangle^{2}e^{\thez it\brad}\psi_{\thez}(t)\|_{H^{2}}\le\ve_{0}+C\ve_{1}^{3}\label{estimate-2}
\end{align}
for some $\de_0$ with $0<\de_{0}<\frac{1}{100}$. \end{prop} 
\noindent The proof of this proposition
constitutes the main part of our analysis. It will be given in Section
\ref{sec-energy}.

To complete the proof of \eqref{claim}, we need to control the $L_{\xi}^{\infty}$-norm. In order to estimate $L_{\xi}^{\infty}$-norm,  we introduce a profile modification as follows: For $\thez \in \{+,-\}$,
\begin{align*}
g_{\thez}(t,\xi):=e^{iB(t,\thez\xi)}e^{\thez it\braxi}\wh{\psi_{\thez}}(t,\xi),
\end{align*}
where $B(t,\theta_{0}\xi)$ is given in \eqref{modified-phase}. \begin{prop}[$L_{\xi}^{\infty}$-estimates]\label{prop-scatt}
Let $\psi\in C([0,T];H^{k})$ satisfy a priori assumption \eqref{assumption-apriori}. Then we get
\begin{align}
\sup_{t_{1}\le t_{2}\in[0,T]}\bra{t_{2}}^{\de_{0}}\normo{\bra{\xi}^{10}\Big(g_{\thez}(t_{2})-g_{\thez}(t_{1})\Big)}_{L_{\xi}^{\infty}}\le\ve_{1}\label{bound-linfty}
\end{align}
for sufficiently small  $\de_{0}>0$ and $\theta_0 \in \{ +,-\}$. 
\end{prop}
 This proposition will be
proved in Section \ref{sec-infty}.
The bound \eqref{bound-linfty} implies that 
the global solution $\psi_{\theta_0}=\Pi_{\theta_0}(D)\psi$ converges to a scattering profile $\phi_{\thez}$ defined by
\[
\phi_{\thez}(\xi):= \mathcal F^{-1}\left(\lim_{t\to\infty}g_{\thez}(t,\xi)\right).
\]
Then, \eqref{bound-linfty} leads us to
\begin{align}
\normo{\braxi^{10}\mathcal{F}\left[\psi_{\thez}(t)-e^{-\thez iB(t,\thez D)}e^{-\thez it\bra{D}}\phi_{\thez}\right]}_{L_{\xi}^{\infty}}\les\bra{t}^{-\de_{0}}\ve_{0}.\label{scatt-decou}
\end{align}
Setting $\phi:=\phi_{+}+\phi_{-}$, \eqref{scatt-decou} implies the
modified scattering  \eqref{scattering}. This completes
the proof of Theorem \ref{mainthm}.

\section{Proof of Proposition \ref{prop-energy}}
\label{sec-energy} In this section, we prove the estimates \eqref{estimate-n}--\eqref{estimate-2},
with a priori assumption \eqref{assumption-apriori}. 
We write \eqref{inteq0} by decomposing the nonlinear term 
\begin{align*}
\begin{aligned}
	\psi_{\theta_0}(t) 
	& =e^{- \theta_0 it\langle D\rangle}\psi_{0,\theta_0}+i\int_{0}^{t}e^{- i\theta_0 (t-s)\langle D\rangle}\Pi_{\theta_0}(D)\Big[(|x|^{-1}*|\psi|^{2})\psi\Big](s)\,ds, \\ 
	&=e^{- \theta_0 it\langle D\rangle}\psi_{0,\theta_0}+i\sum_{\theta_{j}\in\{\pm \},j=1,2,3} \int_0^t e^{-i\theta_0(t-s)}\Pi_{\theta_0}(D)\mathcal{N}(\psi_{\theta_1},\psi_{\theta_2},\psi_{\theta_3})(s) ds, 
\end{aligned}
\end{align*}
where $\theta_0\in\{+,-\}$ and $\psi_{\theta_j}=\Pi_{\theta_j}(D)\psi$ for $j=1,2,3$.
In order to prove Proposition~\ref{prop-energy}, we prove 
\begin{align*}
\| \psi_{\theta_0} \|_{\sum_T^{\theta_0}}
\le \| \psi_{\theta_0}\|_{\sum_0^{\theta_0}} + C \sum_{\theta_{j}\in\{\pm \},j=1,2,3} \| \psi_{\theta_1} \|_{\sum_T^{\theta_1}}\| \psi_{\theta_2} \|_{\sum_T^{\theta_2}}\| \psi_{\theta_3} \|_{\sum_T^{\theta_3}}.
\end{align*}
The proof is based on the energy method in weighted energy space. 

\subsection{Idea of proof}\label{subidea}
For the solution $\psi_{\theta}$  to \eqref{maineq-decou} with $\theta\in\{+,-\}$,
we use the interaction representation of $\psi_{\theta}(t)$ so as to track the scattering states
\begin{align*}
	f_{\theta}(t,x):=e^{\theta it\langle D\rangle}\psi_{\theta}(t,x). 
\end{align*}
 We also use a method of space-time resonance
developed in the works \cite{gemasha2008,gemasha2012-jmpa,gemasha2012-annals,pusa,iopu}.
In addition, we exploit the spinorial null-structures intrinsic in the our main equation \eqref{maineq-decou}.
Then, we can rewrite \eqref{inteq0} in terms of the Fourier transform of
$f_{\theta_{0}}$.  Then, 
\begin{align}
	\begin{aligned} & \widehat{f_{\theta_{0}}}(t,\xi)=\widehat{\psi_{0,\theta_{0}}}(\xi)+ ic_1 \sum_{ \theta_{j}\in\{\pm \},j=1,2,3
}\mathcal I_{\mathbf{\Theta}}(t,\xi),\\
	& \mathcal{I}_{\mathbf{\Theta}}(t,\xi) =\int_0^t\!\! \int_{\mathbb{R}^{3}}\Pi_{\theta_0}(\xi)e^{is{p}_{(\theta_0,\theta_1)}(\xi,\eta)}|\eta|^{-2} \mathcal{F}\langle\psi_{\theta_3},\psi_{\theta_2}\rangle(s,\eta) \widehat{f_{\theo}}(s,\xi-\eta)d\eta ds,
	\end{aligned}
	\label{inteq-f}
\end{align}
where  
we denote 4-tuple of signs by $\mathbf{\Theta}=(\theta_{0},\theta_{1},\theta_2,\theta_3)$ and the resonance function is given by 
\begin{align}\label{function-resonance}
	{p}_{(\theta_0,\theta_1)}(\xi,\eta)=\thez\braxi-\theo\langle\xi-\eta\rangle.
\end{align}
We have to deal with 4 cases of ${\mathbf{\Theta}}$. 
Each ${p}_{(\theta_0,\theta_1)}$ has their own features, so need to be dealt with separately.

Let us take a closer look on the structure
of nonlinear terms. In the course of energy  estimates, we take $\nabla_{\xi}$
and $\nabla_{\xi}^{2}$ to $\mathcal I_{{\mathbf{\Theta}}}$. When the derivatives fall
on the phase $e^{isp_{(\thez,\theo)}}$, we have
\begin{align}
\begin{aligned} & \int_0^t\int_{\mathbb{R}^{3}}\Pi_{\theta_{0}}(\xi)s\textbf{m}(\xi,\eta)e^{is{p}_{(\theta_0,\theta_1)}(\xi,\eta)}|\eta|^{-2}\mathcal{F}\langle\psi_{\theta_3},\psi_{\theta_2}\rangle(s,\eta) \\
 & \qquad\qquad\qquad\qquad\qquad\qquad\qquad\qquad\qquad\qquad\times\widehat{f_{\theo}}(s,\xi-\eta)d\eta ds,
\end{aligned}\nonumber\\
\begin{aligned} & \int_0^t\int_{\mathbb{R}^{3}}\Pi_{\theta_{0}}(\xi)s^{2}\left[\textbf{m}(\xi,\eta)\right]^{2}e^{is{p}_{(\theta_0,\theta_1)}(\xi,\eta)}|\eta|^{-2}\mathcal{F}\langle\psi_{\theta_3},\psi_{\theta_2}\rangle(s,\eta) \\
 & \qquad\qquad\qquad\qquad\qquad\qquad\qquad\qquad\qquad\qquad\times\widehat{f_{\theo}}(s,\xi-\eta)d\eta ds,
\end{aligned}
\label{esti-na-2}
\end{align}
where
\begin{align*}
	\mathbf{m}(\xi,\eta)&=\nabla_{\xi}{p}_{(\theta_0,\theta_1)} (\xi,\eta)=\thez\frac{\xi}{\langle\xi\rangle}-\theo\frac{\xi-\eta}{\langle\xi-\eta\rangle}, \\ 
	\left[\textbf{m}(\xi,\eta)\right]^{2} &= \textbf{m}(\xi,\eta) \otimes \textbf{m}(\xi,\eta).
\end{align*}
Here, we have to bound the singularity $|\eta|^{-2}$ and recover the loss of $s$. To achieve this, we use two null structures suited for given ${\mathbf{\Theta}}$. 
The first one is that a smooth function $\mathbf{m}$ satisfies $\mathbf{m}(\xi,0)=0$ if $\theta_0=\theta_1$, more concretely, 
\begin{equation}
	|\nabla_{\xi}^{n}\nabla_{\eta}^{m}\textbf{m}(\xi,\eta)|\sim\frac{|\eta|^{1-m}}{\jp{\xi}^{n+1}},\qquad\text{ when }|\eta|\ll|\xi|,\label{eq:m bound 1}
	\end{equation}
Thanks to \eqref{eq:m bound 1}, we can think that the multiplier $s\textbf{m}(\xi,\eta)|\eta|^{-2}$
behaves like $|\eta|^{-2}$, as far as the estimates are concerned.
The other is the spinorial null structure \eqref{spinorial} which we can apply when $\theta_0\neq\theta_1$ or $\theta_2\neq\theta_3$. 
Here, \eqref{decay-null} corresponds to \eqref{eq:m bound 1}.

On the other hand, we have to recover the loss of time growth $s^2$ in \eqref{esti-na-2}. It can be compensated thanks to the null structures expect for the case when $\theta_0\neq\theta_1$ and $\theta_2=\theta_3$ where only the spinorial null structure from $\theta_0\neq\theta_1$ can be applied to recover one loss of time growth $s$. 
A key observation is that when $\theta_0=\theta_1$ the resonance functions  
$${p}_{(\theta_0,\theta_1)}(\xi,\eta)=\theta_0(\langle\xi\rangle+\langle \xi-\eta\rangle)$$ 
are non-degenerate, which 
enables us to use the time non-resonance via integration by parts in time variable to recover the remaining loss of $s$. 
Moreover, an additional factor $|\eta|$ essential to close the estimates can be obtained as a by product, which we regard as another null structure, see Lemma \ref{lem:linearpart}. 




\subsection{Estimates for \eqref{estimate-n} and \eqref{estimate-1}}\label{subsection w1}  
It follows from \eqref{NHk} that 
\begin{align*}
 \| \psi_{\theta_0}(t) \|_{H^s} \le \ep_0 + C \ep_1^3 \int_0^t (1+s)^{-1} ds,
\end{align*}
which completes the proof of \eqref{estimate-n}. Let us consider \eqref{estimate-1}. Note that
\[
\|x e^{\thez it\jp D}\psi_{\thez}\|_{H^{2}}\sim\|\langle\xi\rangle^{2}\mathcal{F}(x e^{\thez it\jp D}\psi_{\thez})\|_{L_{\xi}^{2}}\sim\|\braxi^{2}\nabla_{\xi}\widehat{f_{\thez}}\|_{L_{\xi}^{2}}.
\]
From the Duhamel's formula \eqref{inteq-f}, $\nabla_{\xi}\widehat{f_{\thez}}$
satisfies that 
\begin{align*}
\nabla_{\xi}\widehat{f_{\thez}}(t,\xi) & =\nabla_{\xi}\widehat{\psi_{0,\thez}}(\xi)+ ic_1\sum_ {\thej \in \{ \pm\}, \ j=1,2,3}\Big[\mathcal I_{{\mathbf{\Theta}}}^{1}(t,\xi)+\mathcal I_{{\mathbf{\Theta}}}^{2}(t,\xi)+ \mathcal I_{{\mathbf{\Theta}}}^{3}(t,\xi)\Big],
\end{align*}
where
\begin{align*}
\mathcal I_{{\mathbf{\Theta}}}^{1}(t,\xi) & =\int_0^t\int_{\mathbb{R}^{3}}\Pi_{\thez}(\xi)\Pi_{\theo}(\xi-\eta)e^{is{p}_{(\thez,\theo)}(\xi,\eta)}|\eta|^{-2}\\
 & \qquad\qquad\qquad\qquad\qquad\qquad\qquad\times\mathcal{F}\langle\psi_{\theta_3},\psi_{\theta_2}\rangle(\eta)\nabla_{\xi}\widehat{f_{\theta_1}}(\xi-\eta)d\eta ds,\\
 \mathcal I_{{\mathbf{\Theta}}}^{2}(t,\xi) &  = \int_0^t\int_{\mathbb{R}^{3}}\nabla_{\xi}\left[\Pi_{\thez}(\xi)\Pi_{\theo}(\xi-\eta)\right]e^{is{p}_{(\thez,\theo)}(\xi,\eta)}|\eta|^{-2}\\
 & \qquad\qquad\qquad\qquad\qquad\qquad\qquad\times\mathcal{F}\langle\psi_{\theta_3},\psi_{\theta_2}\rangle(\eta)\widehat{f_{\theta_1}}(\xi-\eta)d\eta ds,\\
\mathcal I_{{\mathbf{\Theta}}}^{3}(t,\xi) & =i\int_0^t s\int_{\mathbb{R}^{3}}\Pi_{\thez}(\xi)\Pi_{\theo}(\xi-\eta)\mathbf{m}(\xi,\eta)e^{is{p}_{(\thez,\theo)}(\xi,\eta)}|\eta|^{-2}\\
 & \qquad\qquad\qquad\qquad\qquad\qquad\qquad\times\mathcal{F}\langle\psi_{\theta_3},\psi_{\theta_2}\rangle(\eta)\widehat{f_{\theta_1}}(\xi-\eta)d\eta ds,\\
\end{align*}
and 
\begin{align*}
p_{(\thez,\theo)}(\xi,\eta) & =\thez\langle\xi\rangle-\theo\langle\xi-\eta\rangle,\quad
\mathbf{m}_{(\thez,\theo)}(\xi,\eta)  =\nabla_{\xi}p_{(\thez,\theo)}=\thez\frac{\xi}{\langle\xi\rangle}-\theo\frac{\xi-\eta}{\langle\xi-\eta\rangle}.
\end{align*}
It suffices to show the following estimates: for any ${\mathbf{\Theta}}=(\thez,\theo,\thet,\theth)$,
\[
\normo{\langle\xi\rangle^{2} \mathcal I_{{\mathbf{\Theta}}}^{1}(t,\xi)}_{L_{\xi}^{2}}+\normo{\langle\xi\rangle^{2}\mathcal I_{{\mathbf{\Theta}}}^{2}(t,\xi)}_{L_{\xi}^{2}}+\left\Vert \langle\xi\rangle^{2} \mathcal I_{{\mathbf{\Theta}}}^{3}(t,\xi)\right\Vert _{L_{\xi}^{2}}\les\langle t\rangle^{\de_{0}}\ve_{1}^{3}.
\]

\medskip

\textbf{Estimate for $\mathcal I_{{\mathbf{\Theta}}}^{1}$ and $\mathcal I_{{\mathbf{\Theta}}}^{2}$.} 
Using \eqref{Har Hk} and the a priori assumption \eqref{assumption-apriori}, we estimate 
\begin{align*}
&	\normo{\langle\xi\rangle^{2} \mathcal I_{{\mathbf{\Theta}}}^{1}(t,\xi)}_{L_{\xi}^{2}}\\
&\les \int_0^t \| \mathcal{N}(e^{-\theo it\bra{D}}x e^{\theta_1it\langle D\rangle}  \psi_{\theta_1}(s), \psi_{\theta_2}(s),\psi_{\theta_3}(s)) \|_{H^2} ds\\ 
&\les \int_0^t \|xe^{\theo it\bra{D}}\psi_{\theta_1}(s)\|_{H^2} \left( \|\psi_{\theta_2}(s)\|_{L^6}\|\psi_{\theta_3}(s)\|_{H^2} 
+ \|\psi_{\theta_2}(s)\|_{H^2}\|\psi_{\theta_3}(s)\|_{L^6} \right) ds\\ 
&\les  \int_0^t \langle s\rangle^{-1+\de_{0}}\ve_{1}^{3} ds \les \bra{t}^{\de_0}\ve_1^3.
\end{align*}
Since the Dirac projection operator is bounded \eqref{bound of projection}, we estimate similarly that 
\begin{align*}
&	\normo{\langle\xi\rangle^{2} \mathcal I_{{\mathbf{\Theta}}}^{2}(t,\xi)}_{L_{\xi}^{2}}\\
&\les \int_0^t 
\Big\| \mathcal{N}\Big( e^{-\theo it\bra{D}}(\nabla \; \Pi_{\theta_1})(D) e^{\theta_1 it\langle D\rangle}  \psi_{\theta_1}(s), \psi_{\theta_2}(s),\psi_{\theta_3}(s)\Big) \Big\|_{H^2} \\
&\hspace{4cm} + \left\| (\nabla \; \Pi_{\theta_0})(D) \mathcal{N}(   \psi_{\theta_1}(s), \psi_{\theta_2}(s),\psi_{\theta_3}(s)) \right\|_{H^2} ds\\ 
&\les \int_0^t \|\psi_{\theta_1}(s)\|_{H^2} \left( \|\psi_{\theta_2}(s)\|_{L^6}\|\psi_{\theta_3}(s)\|_{H^2}
+ \|\psi_{\theta_2}(s)\|_{H^2}\|\psi_{\theta_3}(s)\|_{L^6} \right) ds \\ 
&\les \int_0^t \langle s\rangle^{-1+\de_{0}}\ve_{1}^{3} ds \les \bra{t}^{\de_0}\ve_1^3.
\end{align*}

\medskip

\textbf{Estimate for $\mathcal I_{{\mathbf{\Theta}}}^{3}$.} We decompose $|\xi|,|\xi-\eta|,|\eta|$
into dyadic pieces $N_{0},N_{1},N_{2}$, respectively. Then, by the triangle
inequality, we get
\begin{align*}
\|\braxi^{2}\mathcal I_{{\mathbf{\Theta}}}^{3}(t ,\xi)\|_{L_{\xi}^{2}} & \les\sum_{N_{j}\in2^{\mathbb{Z}},j=0,1,2}\bra{N_{0}}^{2}\left\Vert \mathcal I_{{\mathbf{\Theta}},\textbf{N}}^{3}(t,\xi)\right\Vert _{L_{\xi}^{2}},
\end{align*}
where 
\begin{align*}
\mathcal I_{{\mathbf{\Theta}},\textbf{N}}^{3}(t,\xi) & := \int_0^t \!\! \int_{\R^3}
\mathbf{m}_{(\thez,\theo),\textbf{N}}(\xi,\eta) \mathcal{F}{\langle\psi_{{\theta_3}},\psi_{{\theta_2}}\rangle}(\eta)e^{-\theo is\langle\xi-\eta\rangle}\widehat{P_{N_1}f_{\theta_1}}(\xi-\eta)d\eta ds 
\end{align*}
and
\begin{align}\begin{aligned}\label{mthetan}
\textbf{N} & :=(N_{0},N_{1},N_{2}),\\
\mathbf{m}_{(\thez,\theo),\textbf{N}}(\xi,\eta) & := \Pi_{\thez}(\xi)\Pi_{\theo}(\xi-\eta)\mathbf{m}_{(\thez,\theo)}(\xi,\eta)|\eta|^{-2}\\
&\hspace{3cm}\times\rho_{N_{0}}(\xi)\rho_{N_{1}}(\xi-\eta)\rho_{N_{2}}(\eta).
\end{aligned}\end{align}
It suffices to show that
\begin{align*}
\sum_{N_{j}\in2^{\mathbb{Z}},j=0,1,2}\bra{N_{0}}^{2}\normo{\mathcal I_{{\mathbf{\Theta}},\textbf{N}}^{3}(t,\xi)}_{L_{\xi}^{2}}\les\bra{t}^{\de_{0}}\ve_{1}^{3}.
\end{align*}
Using H\"older inequality with the pointwise bound $|\mathbf{m}_{(\thez,\theo),\textbf{N}}(\xi,\eta)| \les N_2^{-1}$, we see that
\begin{align}\label{lowfre esti}
\normo{\mathcal I_{{\mathbf{\Theta}},\textbf{N}}^{3}(t,\xi)}_{L_{\xi}^{2}} 
\les \int_0^t N_2^{-1}\|\rho_{N_{0}}\|_{L^{2}} \left\|P_{N_2}\langle\psi_{\theta_3},\psi_{\theta_2}\rangle(s) \right\|_{L^{2}}\left\|P_{N_1}f_{\theta_1}(s)\right\|_{L^{2}} ds.
\end{align}
From the a priori assumption \eqref{assumption-apriori}, we have 
\begin{align}\label{PNf}
	\left\|P_{N_1}f_{\theta_1}(s)\right\|_{L^{2}} \les N_1^\frac32 \langle N_1\rangle^{-10}\ep_1.
\end{align}
We consider the sum over those $\mathbf{N}$ such that $N_0\le \langle s\rangle^{-2}$ and $N_0\ge \langle s\rangle^{-2}$ in the integrand of $\mathcal I_{\mathbf{\Theta},\textbf{N}}^3$ as follows:
\begin{align*}
	 \mathcal I_{{\mathbf{\Theta}},\textbf{N}}^{3}(t,\xi) &= \int_0^t \left[ \sum_{\{\mathbf{N}: N_0\le \langle s\rangle^{-2} \}}  \Big(\cdots \Big)+ \sum_{\{\mathbf{N}: N_0\ge \langle s\rangle^{-2} \}}\Big(\cdots \Big) \right] ds\\
	&=: \mathcal I_{{\mathbf{\Theta}},\textbf{N}}^{3,1}(t,\xi) + \mathcal I_{{\mathbf{\Theta}},\textbf{N}}^{3,2}(t,\xi).
\end{align*}
Thus, applying \eqref{PNf} and \eqref{norm-two} to \eqref{lowfre esti},  we estimate 
\begin{align*}
	 \bra{N_{0}}^{2}\normo{\mathcal I_{{\mathbf{\Theta}},\textbf{N}}^{3,1}(t,\xi)}_{L_{\xi}^{2}} 
	 & \les \int_0^t s\ve_{1}^{3}\sum_{\{\mathbf{N}: N_0\le \langle s\rangle^{-2} \}}\langle N_{0}\rangle^{2}N_{0}^{\frac{3}{2}}N_{1}^{\frac32}\langle N_{1}\rangle^{-10}N_{2}^{\frac32}\langle N_{2}\rangle^{-10}ds\\
	&\quad  \les \int_0^t s \langle s\rangle^{-3}\ve_{1}^{3} ds \les  \bra{t}^{\de_0}\ve_1^3.
\end{align*}
Besides the pointwise bound, one can verify that 
\begin{align}\label{phase}
		\|	\mathbf{m}_{(\theta_0,\theta_1),\textbf{N}}\|_{\rm CM}
		=\normo{\iint_{\R^{3+3}}	\mathbf{m}_{(\theta_0,\theta_1),\textbf{N}}(\xi,\eta)e^{ix\cdot\xi}e^{iy\cdot\eta}d\xi d\eta}_{L_{x,y}^{1}} 
		 \les N_2^{-1}.
\end{align}
Indeed, we observe that $\mathbf{m}_{(\theta_0,\theta_1),\textbf{N}}$ satisfies the differential inequalities
\begin{align}\label{dif ineq}
	\Big|\nabla_{\xi}^{n}\nabla_{\eta}^{m}\mathbf{m}_{(\theta_0,\theta_1),\textbf{N}}(\xi,\eta)\Big|\les 
	N_{0}^{-n}N_{2}^{-1-m},  \;\;\; \text{ for } N_0 \les N_1. 
	\end{align}
Here, when $N_2\ll N_0\sim N_1$, we used \eqref{eq:m bound 1} for $\theta_0=\theta_1$ or \eqref{decay-null} for $\theta_0\neq\theta_1$.
For $N_0\gg N_1$, we first change variables 
\begin{align*}
	&\normo{\iint_{\R^{3+3}}\mathbf{m}_{(\theta_0,\theta_1),\textbf{N}}(\xi,\eta)e^{ix\cdot\xi}e^{iy\cdot\eta}d\xi d\eta}_{L_{x,y}^{1}}  \\ 
	&\quad = \normo{\iint_{\R^{3+3}}\mathbf{m}_{(\theta_0,\theta_1),\textbf{N}}(\xi,\xi-\eta)e^{ix\cdot\xi}e^{iy\cdot\eta}d\xi d\eta}_{L_{x,y}^{1}}
\end{align*}
and find the differential inequalities 
\begin{align}\label{dif ineq 2}
\Big|\nabla_{\xi}^{n}\nabla_{\eta}^{m}	
\mathbf{m}_{(\theta_0,\theta_1),\textbf{N}}(\xi,\xi-\eta)\Big| \les N_2^{-1}N_0^{-n}N_1^{-m}.
\end{align}
Then, a standard calculation together with \eqref{dif ineq} and \eqref{dif ineq 2}, as we did in \eqref{standard calculation}, yields \eqref{phase}.
Applying the Coifman-Meyer operator estimates, Lemma~\ref{kernel}, with \eqref{phase} we obtain 
\begin{align}
\begin{aligned}\label{Ithree}
& \bra{N_{0}}^{2}\normo{\mathcal I_{{\mathbf{\Theta}},\textbf{N}}^{3,2}(t,\xi)}_{L_{\xi}^{2}}  \\
 & \les \int_0^t s \sum_{\{\mathbf{N}: N_0\ge  \langle s\rangle^{-2}\}}\langle N_{0}\rangle^{2}N_{2}^{-1}\|P_{N_{2}}\langle\psi_{\theta_3},\psi_{\theta_2}\rangle\|_{L_{x}^{\infty}}\|P_{N_{1}}f_{\theta_1}\|_{L_{x}^{2}} ds.
\end{aligned}\end{align}
Then, using the frequency localized estimates \eqref{norm-infty} and \eqref{PNf}, we bound the sum in integrand of \eqref{Ithree} for $N_2\le \langle s\rangle^{-1}$ by  
\begin{align*}
	\int_0^t s \ve_1^3\sum_{\{\mathbf{N}: N_0\ge  \langle s\rangle^{-2}, \ N_2 \le \langle s\rangle^{-1}\}}	\langle N_{0}\rangle^2 N_{2}^{2} N_{1}^{\frac{3}{2}}\langle N_{1}\rangle^{-10} 
	 \les\int_0^t s \langle s\rangle^{-2+\de_{0}}\ve_{1}^{3} ds 
\end{align*}
and the sum in integrand of \eqref{Ithree} for $N_2\ge \langle s \rangle^{-1}$ by 
\begin{align*}
	&\int_0^t s  \ep_1^3\sum_{\{\mathbf{N}: N_0\ge  \langle s\rangle^{-2}, \ N_2 \ge \langle s\rangle^{-1}\}}	\langle N_{0}\rangle^2 N_{2}^{-1} \langle N_{2}\rangle^{-2} N_{1}^{\frac{3}{2}}\langle N_{1}\rangle^{-10} \langle s\rangle^{-3}   ds\\
	& \les\int_0^t s \langle s\rangle^{-2+\de_{0}}\ve_{1}^{3} ds.
\end{align*}

\subsection{Estimates for \eqref{estimate-2}}
In this subsection, we devote to establish \eqref{estimate-2}. By Plancherel's theorem, $\||x|^{2}e^{\thez it\brad}\psi_{\thez}\|_{H^{2}}\sim\left\Vert \braxi^{2}\nabla_{\xi}^{2}\widehat{f_{\thez}}\right\Vert _{L_{\xi}^{2}}$.
 If we take $\nabla_{\xi}^{2}$
to \eqref{inteq-f}, derivatives may fall on  $\wh{f_{\theta_1}}(\xi-\eta)$, $\Pi_{\theta_{0}}(\xi)\Pi_{\theta_{1}}(\xi-\eta),$ or 
$e^{is{\bf p}_{(\theta_{0,}\theta_{1})}}$.   We handle each term case by case. 
The main terms occur when $\nabla_{\xi}^{2}$ falls only on the resonance function, $e^{is{\bf p}_{(\theta_{0,}\theta_{1})}}$ (see \textbf{Case D} below).
Since the other cases are estimated similarly as in the proof of \eqref{estimate-1},  we will omit the details.  

\smallskip

\textbf{Case A:} At least one derivative $\nabla_{\xi}$ falls on $\widehat{f_{\theta_1}}$. \\ 
The proof is essentially same as the proof of \eqref{estimate-1}.
We have only to redo the estimates for $\mathcal I_{\mathbf{\Theta}}^1,\mathcal I_{\mathbf{\Theta}}^2$ and $\mathcal I_{\mathbf{\Theta}}^3$ in subsection~\ref{subsection w1} replacing $\widehat{f_{\theta_1}}$ with $\nabla_\xi \widehat{f_{\theta_1}}$. 

\smallskip

\textbf{Case~B:} $\nabla_{\xi}^{2}$ falls on  $\Pi_{\thez}(\xi)\Pi_{\theo}(\xi-\eta)$. \\ 
We simply use the rough bound
\[
\nabla_{\xi}^{2}\left[\Pi_{\thez}(\xi)\Pi_{\theo}(\xi-\eta)\right]\lesssim1
\]
and follow the argument for $\mathcal I_{\mathbf{\Theta}}^{2}$ in subsection~\ref{subsection w1}. 

\smallskip

\textbf{Case C:} Only one derivative $\nabla_{\xi}$ falls on $e^{is{\bf p}_{(\theta_{0,}\theta_{1})}}$.
Then, the multipliers in integrals correspond to one of the followings:
\begin{enumerate}[(1)]
	\item $s\mathbf{m}_{(\thez,\theo)}(\xi,\eta)\nabla_{\xi}\left(\Pi_{\thez}(\xi)\Pi_{\theo}(\xi-\eta)\right)$.
	\item $s\nabla_{\xi}\mathbf{m}_{(\thez,\theo)}(\xi,\eta)\; \Pi_{\thez}(\xi)\Pi_{\theo}(\xi-\eta)$.
\end{enumerate}
Both can be dealt with by repeating the argument in the proof of  $\mathcal I_{\mathbf{\Theta}}^{3}$ in subsection~\ref{subsection w1}. Actually,
the additional $\xi$-derivative, $\nabla_\xi$, compared to \eqref{mthetan} makes a bound in \eqref{phase} better by giving a $\langle N_0\rangle^{-1}$ factor.

\smallskip

\textbf{Case D:} $\nabla_\xi^2$ falls on $e^{is{\bf p}_{(\theta_{0,}\theta_{1})}}$. Then, we consider the following integral
\begin{align*}
 \langle \xi\rangle^2\mathcal J_{{\mathbf{\Theta}}}(t,\xi) & := \int_0^t  \int_{\R^3} \langle \xi\rangle^2 s^2\left[\mathbf{m}_{(\thez,\theo)}(\xi,\eta) \otimes \mathbf{m}_{(\thez,\theo)}(\xi,\eta)\right](\xi,\eta)|\eta|^{-2}  \\
&\hspace{1cm} \times \Pi_{\thez}(\xi)\Pi_{\theo}(\xi-\eta) e^{is{p}_{(\thez,\theo)}(\xi,\eta)}\widehat{\langle\psi_{\theta_3},\psi_{\theta_2}\rangle}(\eta)\widehat{f_{\theta_1}}(\xi-\eta)d\eta ds
\end{align*}
and prove that 
\begin{align}\label{caseD goal}
\left\| \bra{\xi}^2 \mathcal J_{{\mathbf{\Theta}}}(s,\xi)  \right\|_{L_\xi^2(\R^3)} 
\les C\langle t\rangle^{2\delta_0}\ep_1^3.
\end{align}
Most of all, one needs to handle an extra time growth $s^{2}.$ To compensate the time growth,
 we need to scrutinize the inner product $\jp{\psi_{\theta_3},\psi_{\theta_2}}$. By applying the dyadic decomposition, we may write
\begin{align*}
\mathcal J_{{\mathbf{\Theta}}}(t,\xi)&=\sum_{N_{j}\in2^{\mathbb{Z}},j=0,1,2}\mathcal  J_{{\mathbf{\Theta}},\textbf{N}}(t,\xi),
\end{align*}where
\begin{align*}
	\mathcal J_{{\mathbf{\Theta}},\textbf{N}}(t,\xi) &= \int_0^t s^2\int_{\R^3}\mathfrak{M}_{(\theta_0,\theta_1),\textbf{N}}(\xi,\eta)e^{is{p}_{(\thez,\theo)}(\xi,\eta)}\\
	&\hspace{3cm}\times\widehat{ P_{N_2}\langle\psi_{\theta_3},\psi_{\theta_2}\rangle}(\eta)\widehat{P_{N_{1}}f_{\theta_1}}(\xi-\eta)d\eta ds, 
\end{align*}
and
\begin{align*}
\mathfrak{M}_{(\theta_0,\theta_1),\textbf{N}}(\xi,\eta)
&=\left(\thez\frac{\xi}{\langle\xi\rangle}-\theo\frac{\xi-\eta}{\langle\xi-\eta\rangle}\right)^2 \Pi_{\thez}(\xi)\Pi_{\theo}(\xi-\eta)|\eta|^{-2}\\
&\hspace{5cm}\times\rho_{N_{0}}(\xi)\rho_{N_{1}}(\xi-\eta)\rho_{N_{2}}(\eta).
\end{align*}
We deal with the following three cases separately:
\begin{enumerate}[$(i)$]
\item $\theta_0=\theta_1$,
\item $\theta_0\neq\theta_1$ and $\theta_2\neq\theta_3$,
\item $\theta_0\neq\theta_1$ and $\theta_2=\theta_3$.
\end{enumerate}
\textit{Estimates for (i)}.
As for \eqref{caseD goal}, it suffices to show that
\begin{align}\label{caseD goal2}
\sum_{N_j\in2^{\Z},j=0,1,2}\bra{N_{0}}^{2}\normo{\mathcal  J_{{\mathbf{\Theta}},\textbf{N}}(t)}_{L_{\xi}^{2}} \les\bra{t}^{2\de_{0}}\ve_{1}^{3}.
\end{align}
For $N_0$ with $N_0\le \langle s\rangle^{-2}$ in the integrand of $\mathcal J_{\mathbf{\Theta},\textbf{N}}$, the desired bound can be obtained as in \eqref{lowfre esti}.
For the remaining contribution, we apply the Coifman-Meyer multiplier estimates.
Using the differential inequalities \eqref{eq:m bound 1} and \eqref{decay-null}, one can show as in \eqref{phase} that 
\begin{align}\label{mul M}
\normo{ \mathfrak{M}_{(\theta_0,\theta_0),\textbf{N}} }_{\rm CM} \les 1.
\end{align}
Applying Lemma~\ref{kernel} with \eqref{mul M} and using \eqref{norm-infty}, we estimate that 
\begin{align*}
	 &\bra{N_{0}}^{2}\normo{\mathcal J_{{\mathbf{\Theta}},\textbf{N}}(t)}_{L_{\xi}^{2}}\\
	&\quad\les \int_0^t s^2 \sum_{N_0\ge \langle s\rangle^{-2}} \bra{N_{0}}^{2}\|P_{N_{2}}\langle\psi_{\theta_3},\psi_{\theta_2}\rangle\|_{L_{x}^{\infty}}\|P_{N_{1}}f_{\theta_1}\|_{L_{x}^{2}} ds\\
	&\quad \les \int_0^t s^2 \ve_{1}^{3}\sum_{N_0\ge \langle s\rangle^{-2}}\bra{N_{0}}^{2} \langle N_{1}\rangle^{-10}N_{1}^{\frac{3}{2}}\min\left(N_{2}^{3},\langle N_{2}\rangle^{-2}\langle s\rangle^{-3}\right) ds\\
	& \quad \les \int_0^t\langle s\rangle^{-1+2\de_{0}}\ve_{1}^{3} ds \les \bra{t}^{2\de_0} \ve_1^3.
   \end{align*}

\noindent \textit{Estimates for (ii)}.
We proceed as in the previous case. The sum in \eqref{caseD goal2} for $N_0\le \langle s\rangle^{-1}$ can be obtained as before. In the multiplier estimates, 
since $\theta_0\neq\theta_1$, we can only use the spinorial null structure from $ \Pi_{\thez}(\xi)\Pi_{-\thez}(\xi-\eta)$ to obtain 
\begin{align}\label{mul M2}
\normo{ \mathfrak{M}_{(\theta_0,-\theta_0),\textbf{N}} }_{\rm CM} \les N_2^{-1},
\end{align}
which is worse than \eqref{mul M} when $N_2\ll 1$.
This can be compensated, however, by the null structure in the bilinear form $P_{N_2}\langle\psi_{-\theta_2},\psi_{\theta_2}\rangle$. Indeed, applying Lemma~\ref{kernel} with \eqref{mul M2} and then using \eqref{null-bilinear} instead of \eqref{norm-infty}, we estimate  
\begin{align*}
	&\bra{N_{0}}^{2}\normo{\mathcal J_{{\mathbf{\Theta}},\textbf{N}}(t)}_{L_{\xi}^{2}}\\
   &\quad\les \int_0^t s^2\sum_{N_0\ge \langle s\rangle^{-2}} \bra{N_{0}}^{2}N_2^{-1} \|P_{N_{2}}\langle\psi_{-\theta_2},\psi_{\theta_2}\rangle\|_{L_{x}^{\infty}}\|P_{N_{1}}f_{\theta_1}\|_{L_{x}^{2}} ds \\
   &\quad \les \int_0^t s^2\ve_{1}^{3}\sum_{N_0\ge \langle s\rangle^{-2}}\bra{N_{0}}^{2} \langle N_{1}\rangle^{-10}N_{1}^{\frac{3}{2}}\min\left(N_{2}^{3},\langle N_{2}\rangle^{-2}\langle s\rangle^{-3}\right) ds\\
   &\quad \les \int_0^t\langle s\rangle^{-1+2\de_{0}}\ve_{1}^{3} ds \les \bra{t}^{2\de_0}\ve_1^3.
  \end{align*}

\noindent \textit{Estimates for (iii)}.
It remains to handle the case $\thez\neq\theo$ and $\thet=\theth$. A key observation is that the resonance function 
\begin{align*}
	P_{(\theta_0,\theta_1)}(\xi,\eta) = \theta_0 \left( \bra{\xi}+\bra{\xi-\eta}\right) 
\end{align*}
is non-degenerate which enables to perform an integration by parts in time variables
or exploit the time non-resonancy. 
In this procedure, the condition $\theta_2=\theta_3$ plays an important role to obtain a factor $N_2$ essential to bound the singularity.
We begin with writing the integral $\mathcal J_{\mathbf{\Theta},\textbf{N}}$  as 
\begin{align*}
 &\mathcal J_{{\mathbf{\Theta}},\textbf{N}}(t,\xi) \\
 &=\int_0^t s^2\int_{\R^3}\mathfrak{M}_{(\theta_0,-\theta_0),\textbf{N}}(\xi,\eta)e^{is{p}_{(\thez,\theo)}(\xi,\eta)}\widehat{P_{N_2}\langle\psi_{\theta_3},\psi_{\theta_2}\rangle}(s,\eta)\widehat{P_{N_{1}}f_{\theta_1}}(s,\xi-\eta)d\eta ds\\
 &=-\thez i \int_0^t s^2\int_{\R^3} \mathfrak{M}_{(\theta_0,-\theta_0),\textbf{N}}(\xi,\eta) \frac{\partial_s e^{is{p}_{(\thez,\theo)}(\xi,\eta)}}{\bra{\xi} + \bra{\xi-\eta}}\widehat{P_{N_2} \langle\psi_{\theta_3},\psi_{\theta_2}\rangle}(s,\eta)\\
 &\hspace{9cm}\times\widehat{P_{N_{1}}f_{\theta_1}}(s,\xi-\eta)d\eta ds.
\end{align*}
From now, we fix $\theta_0=\theta_2=+$ for simplicity and denote 
${p}_{(+,-)}$ by $p$.
Let us define a multiplier $\widetilde{\mathfrak M}_{\textbf{N}}$ by
\begin{align*}
\widetilde{\mathfrak M}_{\textbf{N}}(\xi,\eta,\sigma)&=\left(\frac{\xi}{\langle\xi\rangle}+\frac{\xi-\eta}{\langle\xi-\eta\rangle}\right)^2 \Pi_{+}(\xi)\Pi_{-}(\xi-\eta)\frac{|\eta|^{-2}}{\bra{\xi}+ \bra{\xi-\eta}}\\
&\hspace{5cm}\times\rho_{N_{0}}(\xi)\rho_{N_{1}}(\xi-\eta)\rho_{N_{2}}(\eta).
\end{align*}
Performing the integration by parts in time, we obtain
\begin{align*}
	\begin{aligned} & \mathcal J_{{\mathbf{\Theta}, \textbf{N}}}(t,\xi)\\
		& =t^{2}\int_{\mathbb{R}^{3}}\widetilde{\mathfrak M}_{\textbf{N}}(\xi,\eta)e^{it  \left( \bra{\xi}+\bra{\xi-\eta}\right) } \mathcal{F}\left( P_{N_2}\bra{ \psi_{+},\psi_{+}} \right) (t,\eta)\widehat{P_{N_{1}}f_{-}}(t,\xi-\eta) d\eta\\
		& \;\;+\int_{0}^{t}2s\int_{\mathbb{R}^{3}}\widetilde{\mathfrak M}_{\textbf{N}}(\xi,\eta)e^{is  \left( \bra{\xi}+\bra{\xi-\eta}\right) } \mathcal{F}\left( P_{N_2}\bra{ \psi_{+},\psi_{+}} \right) (s,\eta)\widehat{P_{N_{1}}f_{-}}(s,\xi-\eta) d\eta ds\\
		&\;\; +\int_{0}^{t}s^{2}\int_{\mathbb{R}^{3}}\widetilde{\mathfrak M}_{\textbf{N}}(\xi,\eta)e^{is  \left( \bra{\xi}+\bra{\xi-\eta}\right) }\partial_{s}  \left( P_{N_2}\bra{ \psi_{+},\psi_{+}} \right)(s,\eta) 
		\widehat{P_{N_{1}}f_{-}}(s,\xi-\eta) d\sigma d\eta\,ds\\
		&\;\; +\int_{0}^{t}s^{2}\int_{\mathbb{R}^{3}}\widetilde{\mathfrak M}_{\textbf{N}}(\xi,\eta)e^{is  \left( \bra{\xi}+\bra{\xi-\eta}\right) }\mathcal{F}\left( P_{N_2}\bra{ \psi_{+},\psi_{+}} \right) (s,\eta)\widehat{\partial_{s}P_{N_{1}}f_{-}}(s,\xi-\eta) d\eta\,ds\\
		& =:\mathcal L_{ \textbf{N}}^{1}(t,\xi)+\mathcal L_{ \textbf{N}}^{2}(t,\xi)+\mathcal L_{ \textbf{N}}^{3}(t,\xi)+\mathcal L_{\textbf{N}}^{4}(t,\xi).
	\end{aligned}
\end{align*}
The first two terms can be dealt with analogously to the 
estimates for $\mathcal I_{\mathbf{\Theta}}^{3}$ in the previous subsection by using the following multiplier bounds:
\begin{align}\label{eq:bound m'}
		\left\| \iint e^{ix\cdot \xi} e^{iy\cdot \eta}\widetilde{\mathfrak M}_{\textbf{N}}(\xi,\eta) d\xi d\eta  \right\|_{L_{x,y}^1} &\lesssim \min\left(  \langle N_1\rangle^{-1}, \langle N_2\rangle^{-1} \right) {N_{2}^{-1}}.
\end{align}
Next, we consider $\mathcal L_{ \textbf{N}}^{3}(t,\xi)$. First, taking the time derivative to \eqref{maineq-decou} yields that
\begin{align*}
\partial_{s} \left[\bra{ \psi_{+},\psi_{+}} \right]
	&= i\Big( \bra{\bra{D}\psi_+, \psi_+} -  \bra{\psi_+, \bra{D}\psi_+} \Big)  \\
	&\quad + \Big\langle\psi_+, ic_1\Pi_{+}(D)\mathcal{N}({\psi,\psi,\psi}) \Big\rangle 
	+ \Big\langle ic_1\Pi_{+}(D)\mathcal{N}({\psi,\psi,\psi}),\psi_+ \Big\rangle.
\end{align*}
Plugging this into $\mathcal L_{ \textbf{N}}^{3}(t,\xi)$, we suffice to consider the following two integrals:
\begin{align*}
	\mathcal L_{ \textbf{N}}^{3,1}(t,\xi)&= \int_{0}^{t}s^{2}\int_{\mathbb{R}^{3}}\widetilde{\mathfrak M}_{\textbf{N}}(\xi,\eta)e^{is \left( \bra{\xi}+\bra{\xi-\eta}\right) } \widehat{P_{N_{1}}f_{-}}(s,\xi-\eta)\\
	     &\hspace{1cm}\times \mathcal F  P_{N_2}\left(\bigbra{\bra{D}\psi_{+},\psi_{+}} - \bigbra{\psi_{+},\bra{D}\psi_{+}} \right)(s,\eta) d\eta\,ds,\\
	\mathcal L_{ \textbf{N}}^{3,2}(t,\xi)&= \int_{0}^{t}s^{2}\int_{\mathbb{R}^{3}}\widetilde{\mathfrak M}_{\textbf{N}}(\xi,\eta)e^{is \left( \bra{\xi}+\bra{\xi-\eta}\right) } \widehat{P_{N_{1}}f_{-}}(s,\xi-\eta)\\
	     &\hspace{2.5cm}\times\mathcal F  P_{N_2}\Big\langle\psi_+, \Pi_{+}(D)\mathcal{N}({\psi,\psi,\psi}) \Big\rangle  (s,\eta) d\eta\,ds.
\end{align*}
To prove \eqref{caseD goal}, we have to show that 
\begin{align*}
	\sum_{\mathbf{N}} \langle N_0 \rangle^2 \left(  \|\mathcal L_{ \textbf{N}}^{3,1}(t,\xi)\|_{L^2} + \|\mathcal L_{ \textbf{N}}^{3,2}(t,\xi)\|_{L^2} \right) \les \langle t \rangle^{2\delta_0} \ep_1^3.
\end{align*}
We can easily estimate the sum over those $\mathbf{N}$ such that $N_0\le \langle s\rangle^{-2}$ in the integrand of $\mathcal L_{\textbf{N}}^{3,1}$ as in \eqref{lowfre esti}.
For the remaining contribution, we apply the Coifman-Meyer estimates with \eqref{eq:bound m'}.
For $\mathcal L_{ \textbf{N}}^{3,1}$, we estimate 
\begin{align}\label{eq:L31-esti}
	\begin{aligned}
&\langle N_0 \rangle^2  \|\mathcal L_{ \textbf{N}}^{3,1}(t,\xi)\|_{L^2} 	\\
&\les   \langle N_0 \rangle^2  \int_{0}^{t}s^{2} \sum_{\{ \mathbf{N} : N_0\ge\langle s\rangle^{-2}\} } N_2^{-1}  	\normo{P_{N_1}\psi_{-}(s)}_{L^2}\\
& \hspace{2cm} \times \left\|P_{N_2}\left(\bigbra{\bra{D}\psi_{+}(s),\psi_{+}(s)} -  \bigbra{\psi_{+}(s),\bra{D}\psi_{+}(s)} \right) \right\|_{L^{\infty}}ds
\end{aligned}
\end{align}
Here, we take advantage of the structure in $L^\infty$-norm. Indeed, by using \eqref{eq:linearpart}, we bound the sum in the above by  
\begin{align*}
\eqref{eq:L31-esti}&\les \ve_1^3 \int_{0}^{t}s^{2} \sum_{\{ \mathbf{N} : N_0\ge\langle s\rangle^{-2}\} } \langle N_0 \rangle^2 N_1^{\frac32}\bra{N_1}^{-10}\min \left( \bra{N_2}^{-2}\bra{s}^{-3+\frac{\de_0}{10}}, N_2^3\right)ds\\
&\les \ve_1^3 \int_0^t s^2  \left( \sum_{N_2 \le \bra{s}^{-1}}N_2^3 + \sum_{N_2 \ge \bra{s}^{-1}}\bra{N_2}^{-2}\bra{s}^{-3+\frac{\de_0}{10}} \right) ds\\
&\les \bra{t}^{2\de_0}\ve_1^3.
\end{align*}
For $\mathcal L_{ \textbf{N}}^{3,2}$, using \eqref{esti-g-1} we estimate 
\begin{align*}
	\begin{aligned} & \langle N_0 \rangle^2 \left\| \mathcal L_{ \textbf{N}}^{3,2}(t,\xi)\right\|_{L^{2}}\\
		&\les \int_{0}^{t}s^{2}\sum_{\{ \mathbf{N} : N_0\ge\langle s\rangle^{-2}\} }  \langle N_0 \rangle^2 	\normo{P_{N_1}\psi_{-}(s)}_{L^2} 
		\left\|  P_{N_2}\Big\langle\psi_+, \Pi_{+}(D)\mathcal{N}({\psi,\psi,\psi}) \Big\rangle(s)  \right\|_{L^\infty} ds\\
		&\les \ve_1^5\int_{0}^{t}s^{2}\sum_{\{ \mathbf{N} : N_0\ge\langle s\rangle^{-2}\} }  \langle N_0 \rangle^2 N_1^{\frac32}\bra{N_1}^{-10}\min \left(N_2^3 \bra{s}^{-1}, \bra{N_2}^{-2}\bra{s}^{-4+\frac43\de_0} \right)ds\\
		&\les \ve_1^3 \int_0^t s^2  \left( \sum_{N_2 \le \bra{s}^{-\frac34}}N_2^3\bra{s}^{-1} + \sum_{N_2 \ge \bra{s}^{-\frac34}}N_2^{-1}\bra{N_2}^{-2}\bra{s}^{-4+\frac43\de_0} \right) ds\\
		&\les \bra{t}^{2\de_0}\ve_1^3.
	\end{aligned}
\end{align*}
Finally, consider $\mathcal L_{\textbf{N}}^{4}$. The time derivative of $P_{N_1}f_{-}$ is given by 
\begin{align*}
	\partial_{s} P_{N_1}f_{-}(s)=e^{- is\bra{D}}P_{N_1}\Pi_{-}(D)\mathcal{N}(\psi,\psi,\psi)(s).
\end{align*}
Since the Dirac projection operator is bounded in $L^2$, using \eqref{esti-g-2}, we have 
\begin{align*}
\left\| \partial_{s} P_{N_1}f_{-}(s) \right\|_{L^2}	\les \min\left( N_1^\frac32\langle s\rangle^{-\frac52}, \langle N_1\rangle^{-10} \langle s\rangle^{-1}  \right)\ep_1^3.
\end{align*}
With the help of an additional time decay $\langle s\rangle^{-1}$ compared to \eqref{PNf}, we can show as before that 
\begin{align*}
	\sum_{\mathbf{N}} \langle N_0 \rangle^2   \|\mathcal L_{ \textbf{N}}^{4}(t,\xi)\|_{L^2}  \les \langle t \rangle^{2\delta_0} \ep_1^3.
\end{align*}
Indeed, for the sum over $N_0\le \langle s\rangle^{-2}$ in the integrand can be estimated by using the H\"older inequality, while for $N_0\ge \langle s\rangle^{-2}$ applying the Coifman-Meyer estimates, Lemma~\ref{kernel}, with \eqref{eq:bound m'}.

\section{Modified asymptotic states}\label{sec-infty}

\global\long\def\freq{{(\xi,\eta,\sigma)}}%

This section is devoted to proving  Proposition \ref{prop-scatt}.
We assume the a priori bound \eqref{assumption-apriori} and we will justify 
the phase correction for modified scattering profile
\begin{align*}
\left\{ \begin{aligned}B(t,\xi) & =B_{+}(t,\xi)+B_{-}(t,\xi),\\
B_{\theta}(t,\xi) & =\frac{c_1}{(2\pi)^{3}}\int_{0}^{t}\int_{\mathbb{R}^{3}}\left|\frac{\xi}{\langle\xi\rangle}+\theta\frac{\sigma}{\langle\sigma\rangle}\right|^{-1}\abs{\wh{\psi_{\theta}}(\sigma)}^{2}d\sigma\frac{\rho(s^{-\frac{3}{k}}\xi)}{\langle s\rangle}ds,
\end{aligned}
\right.
\end{align*}
and
\[
g_{\thez}(t,\xi)=e^{iB(t,\thez\xi)}e^{\thez it\braxi}\wh{\psi_{\thez}}(t,\xi)
\]
where $\theta, \thez \in \{ \pm \}$. The phase correction $B(t,\xi)$ is required to remedy the nonlinear interaction when the resonance function is degenerate, specifically, the combinations of sign satisfies \eqref{Theta0022}. The explicit form of $B(t,\xi)$ is derived from the computation via the Taylor expansion in the fourier space. 

To prove \eqref{bound-linfty}, we will follow the argument in \cite{pusa}. We show that if $t_{1}\le t_{2}\in[M-2,2M]\cap[0,T]$
for a dyadic number $M\in2^{\mathbb{N}}$, then for some $\de_{0}>0$,
\begin{align}
\normo{\braxi^{10}\Big(g_{\thez}(t_{2},\xi)-g_{\thez}(t_{1},\xi)\Big)}_{L_{\xi}^{\infty}}\les M^{-\de_{0}}\ve_{1}^{3}.\label{goal-modi}
\end{align}
Unlike the representation $\mathcal I_{\mathbf{\Theta}}$ of \eqref{inteq-f}, we abbreviate the integrand of time integral in the nonlinear part of \eqref{inteq-f} by
\begin{align*}
	\mathcal{I}_{\mathbf{\Theta}}(s,\xi)
	& :=\int_{\mathbb{R}^{3}}\Pi_{\theta_0}(\xi)e^{is{p}_{(\theta_0,\theta_1)}(\xi,\eta)}|\eta|^{-2} \mathcal{F}\langle\psi_{\theta_3},\psi_{\theta_2}\rangle(s,\eta) \widehat{f_{\theo}}(s,\xi-\eta)d\eta  \\ 
	&=\iint_{\mathbb{R}^{3+3}}\Pi_{\theta_0}(\xi)e^{is \left( {p}_{(\theta_0,\theta_1)}(\xi,\eta) -\theta_2 \langle \eta+\sigma\rangle+\theta_3\langle \sigma\rangle \right) }|\eta|^{-2} \\ 
	&\qquad\qquad\qquad\qquad \times \Big\langle\widehat{f}_{\theta_3}(s,\sigma),\widehat{f}_{\theta_2}(s,\eta+\sigma)\Big\rangle \widehat{f_{\theo}}(s,\xi-\eta)d\sigma d\eta. \end{align*}
By the change of  variables as 
\[
\sigma \longmapsto \xi-\eta +\sigma, \qquad \eta \longmapsto -\eta,
\]
we obtain 
\begin{align*}
&\mathcal I_{{\mathbf{\Theta}}}(s,\xi)=\iint_{\R^{3+3}}e^{is\phasep\freq}\Pi_{\thez}(\xi)|\eta|^{-2}\wh{f_{\theo}}(s,\xi+\eta)\\
&\hspace{5cm}\times\bigbra{\wh{f_{\theth}}(s,\xi+\eta+\sigma),\wh{f_{\thet}}(s,\xi+\sigma)}d\eta d\sigma.
\end{align*}
where, in the view of \eqref{function-resonance}, the resonance function is given
\[
{p}_{{\mathbf{\Theta}} } := \thez\braxi-\theo\langle\xi+\eta\rangle-\thet\langle\xi+\sigma\rangle+\theth\bra{\xi+\eta+
	\sigma},\;\;( {\mathbf{\Theta}} =(\thez,\theo,\thet,\theth)).
\]
 Among the contribution $\mathcal I_{{\mathbf{\Theta}}}$'s, the cases when $\theta_0=\theta_1$ and $\theta_2=\theta_3$
are critical in the sense of scattering where the phase correction
$e^{iB(t,\thez\xi)}$ is required to the asymptotic formula.

We first decompose $\mathcal I_{{\mathbf{\Theta}}}$ dyadically in terms of $|\eta|\sim L$ into 
\begin{align*}
\begin{aligned} 
& \mathcal I_{{\mathbf{\Theta}}}(s,\xi)\\
 &=ic_{1}\iint_{\R^{3+3}}e^{is\phasep\freq}\Pi_{\thez}(\xi)|\eta|^{-2}\wh{f_{\theo}}(s,\xi+\eta)\\
 &\hspace{5cm}\times\bigbra{\wh{f_{\theth}}(s,\xi+\eta+\sigma),\wh{f_{\thet}}(s,\xi+\sigma)}d\eta d\sigma\\
 & =\mathcal I_{{\mathbf{\Theta}},L_0}(s,\xi)+\sum_{L\in2^{\Z},L>L_{0}}\mathcal I_{{\mathbf{\Theta}},L}(s,\xi),
\end{aligned}
\end{align*}
where
\begin{align*}
\mathcal I_{{\mathbf{\Theta}},L_0}(s,\xi) & :=ic_{1}\iint_{\R^{3+3}}e^{is\phasep\freq}\Pi_{\thez}(\xi)|\eta|^{-2}\rho_{\le L_{0}}(\eta)\wh{f_{\theo}}(s,\xi+\eta)\\
 & \qquad\qquad\qquad\qquad\qquad\qquad\times\bigbra{\wh{f_{\theth}}(s,\xi+\eta+\sigma),\wh{f_{\thet}}(s,\xi+\sigma)}d\eta d\sigma,\\
\mathcal I_{{\mathbf{\Theta}},L}(s,\xi) & :=ic_{1}\iint_{\R^{3+3}}e^{is\phasep\freq}\Pi_{\thez}(\xi)|\eta|^{-2}\rho_{L}(\eta)\wh{f_{\theo}}(s,\xi+\eta)\\
 & \qquad\qquad\qquad\qquad\qquad\qquad\times\bigbra{\wh{f_{\theth}}(s,\xi+\eta+\sigma),\wh{f_{\thet}}(s,\xi+\sigma)}d\eta d\sigma,
\end{align*}
and $L_{0}\in2^{\Z}$ such that
\begin{align}
L_{0}\sim M^{\left(-\frac{3}{4}+\frac{1}{30}\right)}.\label{eq-lzero}
\end{align}
Plugging the above decomposition of $\mathcal I_{\mathbf{\Theta}}$ into $g_{\thez}$, we have for $\theta_0\in\{\pm\}$,
\begin{align}
	\begin{aligned}\label{eq-deriv-g}
 & \partial_{s}g_{\thez}(s,\xi)= e^{-iB(s,\thez\xi)}\\
 & \quad\times\left[\sum_{\substack{\theta_1,\theta_2,\theta_3\in\{\pm\}}}\left(\mathcal I_{{\mathbf{\Theta}},L_0}(s,\xi)+\sum_{L>L_{0}} \mathcal I_{{\mathbf{\Theta}},L}(s,\xi)\right)-i\left[\partial_{s}B(s,\thez\xi)\right]\wh{f_{\thez}}(s,\xi)\right].
 \end{aligned}
\end{align}
In the estimates of \eqref{eq-deriv-g}, the cancellation effect from  $\partial_t B(t,\thez\xi)$ play a role only when $\theta_0=\theta_1$ and $\theta_2=\theta_3$ denoted by 
\begin{align}
\label{Theta0022} {\mathbf{\Xi}}:=(\thez,\thez,\thet,\thet)
\end{align}
so, we divide the cases into $\mathbf{\Theta} = \mathbf{\Xi}$ and $\mathbf{\Theta} \neq \mathbf{\Xi}$.
Thus, to prove \eqref{goal-modi}, it suffices to show that for $\xi$ with 
$|\xi|\sim N_0 \in 2^{\Z}$,
\begin{align}
\sum_{\theta_1,\theta_2,\theta_3\in\{\pm\}}\left|\int_{t_{1}}^{t_{2}}e^{-iB(s,\thez\xi)}\mathcal I_{{\mathbf{\Theta}},L_0}(s,\xi)\left(1-\rho(s^{-\frac3k}\xi)\right)ds\right| & \les\ve_{1}^{3}M^{-\de_{0}}\bra{N_0}^{-10},\label{eq-modi-part1}
\end{align}
\begin{align}
	\begin{aligned}
&\left|\int_{t_{1}}^{t_{2}}e^{-iB(s,\thez\xi)}\left[\mathcal I_{\mathbf{\Xi},L_0}(s,\xi)\rho(s^{-\frac3k}\xi)-i\partial_{s}B(s,\thez\xi)\wh{f_{\thez}(s,\xi)}\right]ds\right|\\
&\qquad \les\ve_{1}^{3}M^{-\de_{0}}\bra{N_0}^{-10},\label{eq-modi-part2}
\end{aligned}
\end{align}
\begin{align}
&\left|\int_{t_{1}}^{t_{2}}e^{-iB(s,\thez\xi)}\sum_{{\theta_1,\theta_2,\theta_3\in\{\pm\}, \mathbf{\Theta}\neq \mathbf{\Xi}}}\mathcal I_{{\mathbf{\Theta}},L_0}(s,\xi)\rho(s^{-\frac3k}\xi)ds\right| \les\ve_{1}^{3}M^{-\de_{0}}\bra{N_0}^{-10},\label{eq-modi-part2-2}
\end{align}
and
\begin{align}
\sum_{\theta_1,\theta_2,\theta_3\in\{\pm\}}\left|\int_{t_{1}}^{t_{2}}e^{-iB(s,\thez\xi)}\sum_{L>L_{0}}\mathcal I_{{\mathbf{\Theta}},L}(s,\xi)ds\right| & \les\ve_{1}^{3}M^{-\de_{0}}\bra{N_0}^{-10}.\label{eq-modi-part3}
\end{align}
The phase modification term in \eqref{eq-modi-part2} will cancel a possible resonance in $\mathcal I_{\mathbf{\Theta},L_0}$ later.

\smallskip

\textbf{Proof of \eqref{eq-modi-part1}.}  When
$N\ll M^{\frac3k}$, the integrand in \eqref{eq-modi-part1}
vanishes, so we may assume $M^{\frac3k}\les N$. It suffices to
show the following bound
\begin{align*}
\Big|\mathcal I_{{\mathbf{\Theta}},L_0}(s,\xi)\Big|\les\ve_{1}^{3}M^{-(1+\de_{0})}\bra{N}^{-10}.
\end{align*}
We further decompose $\mathcal I_{{\mathbf{\Theta}},L_0}$ dyadically in $\sigma$ variable as follows:
\begin{align*}
\mathcal I_{{\mathbf{\Theta}},L_0}(s,\xi)=ic_{1}\sum_{L_{1}\le L_{0}+10}\mathcal J_{L_{1}}(s,\xi),
\end{align*}
where
\begin{align*}
 \mathcal J_{L_{1}}(s,\xi)& =\iint_{\R^{3+3}}e^{is\phasep\freq}\Pi_{\thez}(\xi)|\eta|^{-2}\rho_{L_{1}}(\sigma)\rho_{\le L_{0}}(\eta)\wh{f_{\theta_1}}(s,\xi+\eta)\\
 & \qquad\qquad\qquad\qquad\qquad\qquad\times\bigbra{\wh{f_{\theta_3}}(s,\xi+\eta+\sigma),\wh{f_{\theta_2}}(s,\xi+\sigma)}d\eta d\sigma.
\end{align*}
Then, by a priori assumption \eqref{assumption-apriori}, we estimate
\begin{align*}
 & \left|\mathcal J_{L_{1}}(s,\xi)\right|\\
 & \;\;\les L_{1}^{-2}\bra{\xi}^{-10}\iint_{\R^{3+3}}\left|\bra{N}^{10}\Pi_{\thez}(\xi)\rho_{L_{1}}(\sigma)\rho_{\le L_{0}}(\eta)\wh{f_{\theta_1}}(s,\xi+\eta)\right.\\
 & \left.\qquad\qquad\qquad\qquad\qquad\qquad\qquad\times\bigbra{\wh{\psi_{\theta_3}}(s,\xi+\eta+\sigma),\wh{\psi_{\theta_2}}(s,\xi+\sigma)}\right|d\eta d\sigma\\
 & \;\;\les L_{1}^{-2}\bra{N}^{-10}L_{1}^{3}\|\bra{\xi}^{10}\wh{f_{\theta_1}}\|_{L_{\xi}^{\infty}}\|\psi_{\theta_2}\|_{H^{10}}\|\psi_{\theta_3}\|_{H^{10}}\\
 & \;\;\les L_{1}M^{2\de_{0}}\bra{N}^{-10}\ve_{1}^{3}.
\end{align*}
On the other hand, by H\"older inequality, we get
\begin{align*}
\left|\mathcal J_{L_{1}}(s,\xi)\right|\;\;\les L_{1}^{-\frac{1}{2}}N^{-k}\prod_{j=1}^{3}\|f_{\thej}\|_{H^{k}}\les L_{1}^{-\frac{1}{2}}M^{-3}M^{3\de_{0}}\ve_{1}^{3}.
\end{align*}
Above two estimates induce that
\begin{align*}
 &\sum_{L_{1}\le L_{0}+10}\Big|\mathcal J_{L_{1}}(s,\xi)\Big|\\
 & \les \sum_{L_{1}\le M^{-(1+3\de_{0})}}L_{1}M^{2\de_{0}}\bra{N_0}^{-10}\ve_{1}^{3}+\sum_{L_{1}>M^{-(1+3\de_{0})}}L_{1}^{-\frac{1}{2}}M^{-3+3\de_{0}}\ve_{1}^{3}\\
 & \les M^{-(1+\de_{0})}\bra{N_0}^{-10}\ve_{1}^{3}.
\end{align*}
This completes the proof of \eqref{eq-modi-part1}.

\smallskip

\textbf{Proof of \eqref{eq-modi-part2}.} Due to the cut-off $\rho(s^{-\frac3k}\xi)$,
we may assume $N_0\le M^{\frac3k}$. It suffices to bound the
integrand by 
\begin{align*}
\Big|\mathcal I_{{\mathbf{\Xi}},L_0}(s,\xi)-i\partial_{s}B(s,\thez\xi)\wh{f_{\thez}}(s,\xi)\Big| & \les\ve_{1}^{3}M^{-1-\de_{0}}\bra{N_0}^{-10}.
\end{align*}
for $\mathbf{\Xi} = (\thez, \thez, \thet,\thet)$.
Let us observe that
\begin{align*}
{p}_{\mathbf{\Xi}}\freq & =\thez\Big(\braxi-\bra{\xi+\eta}\Big) - \thet\Big(\bra{\xi+\sigma}-\bra{\xi+\eta+\sigma}\Big)\\
 & =\left(-\thez\frac{|\eta|^{2}+2\eta\cdot\xi}{\braxi+\bra{\xi+\eta}}-\thet\frac{-|\eta|^{2}-2\eta\cdot(\xi+\sigma)}{\bra{\xi+\sigma}+\bra{\xi+\eta+\sigma}}\right)\\
 & =-\eta\cdot\left(\thez\frac{\xi}{\braxi}-\thet\frac{\xi+\sigma}{\bra{\xi+\sigma}}\right)+O\left(|\eta|^{2}\right)\\
 & =:{q}_{(\thez,\thet)}(\xi,\eta,\sigma)+O\left(|\eta|^{2}\right).
\end{align*}
We now set
\begin{align*}
 & \mathcal I_{(\thez,\thet)}(s,\xi)\\
 & =ic_{1}\iint_{\R^{3+3}}e^{is{q}_{(\thez,\thet)}\freq}\Pi_{\thez}(\xi)\rho_{\le L_{0}}(\eta)|\eta|^{-2}\\
 & \qquad\qquad\qquad\qquad\times\wh{f_{\thez}}(\xi+\eta)\bigbra{\wh{f_{\thet}}(\xi+\eta+\sigma),\wh{f_{\thet}}(\xi+\sigma)}d\eta d\sigma.
\end{align*}
Then we estimate
\begin{align*}
 & \left|\mathcal I_{{\mathbf{\Xi}},L_0}(s,\xi)-\mathcal I_{(\thez,\thet)}(s,\xi)\right|\\
 & \les\iint_{\R^{3+3}}s\left|{p}_{{\mathbf{\Xi}}}(\xi,\eta,\sigma)-{q}_{(\thez,\thet)}(\xi,\eta,\sigma)\right|\\
 & \qquad\qquad\times\left|\Pi_{\thez}(\xi)\rho_{\le L_{0}}(\eta)|\eta|^{-2}\wh{f_{\thez}}(\xi+\eta)\bigbra{\wh{f_{\thet}}(\xi+\eta+\sigma),\wh{f_{\thet}}(\xi+\sigma)}\right|d\eta d\sigma\\
 & \les M\iint_{\R^{3+3}}\left|\Pi_{\thez}(\xi)\rho_{\le L_{0}}(\eta)\wh{f_{\thez}}(\xi+\eta)\bigbra{\wh{f_{\thet}}(\xi+\eta+\sigma),\wh{f_{\thet}}(\xi+\sigma)}\right|d\eta d\sigma\\
 & \les ML_{0}^{3}\normo{f_{\thet}}_{L_{x}^{2}}^{2}\normo{\wh{f_{\thez}}}_{L_{\xi}^{\infty}}\\
 & \les M^{\left(-\frac{5}{4}+\frac{1}{10}+2\de_{0}\right)}\ve_{1}^{3},
\end{align*}
where we used \eqref{eq-lzero} in the last inequality. Next, we approximate $\mathcal I_{(\thez,\thet)}(s,\xi)$ by $\wt{\mathcal I_{(\thez,\thet)}}(s,\xi)$ which is defined by
\begin{align*}
\wt{\mathcal I_{(\thez,\thet)}}(s,\xi)&:=ic_{1}\int_{\R^{3+3}}e^{is{q}_{(\thez,\thet)}\freq}\Pi_{\thez}(\xi)\rho_{\le L_{0}}(\eta)|\eta|^{-2}\\
&\hspace{5cm}\times\wh{f_{\thez}}(\xi)\left|\wh{f_{\thet}}(\xi+\sigma)\right|^{2}d\eta d\sigma.
\end{align*}
Set $R=L_0^{-1}$. Then we estimate for $\theta=\theta_0$ or $\theta_2$,
\begin{align*}
 & \left|\wh{f_{\theta}}(\zeta+\eta)-\wh{f_{\theta}}(\zeta)\right|\\
 & \les\left|\wh{\rho_{>R}f_{\theta}}(\zeta+\eta)-\wh{\rho_{>R}f_{\theta}}(\zeta)\right|+\left|\wh{\rho_{\le R}f_{\theta}}(\zeta+\eta)-\wh{\rho_{\le R}f_{\theta}}(\zeta)\right|\\
 & \les\normo{\wh{\rho_{>R}f_{\theta}}}_{L_{\xi}^{\infty}}+L_{0}\normo{\nabla_{\xi}\wh{\rho_{\le R}f_{\theta}}}_{L_{\xi}^{\infty}}\\
 & \les R^{-\frac{1}{2}}\normo{\bra{x}^{2}f_{\theta}}_{L_{x}^{2}}+R^{\frac{1}{2}}L_{0}\normo{\bra{x}^{2}f_{\theta}}_{L_{x}^{2}}\\
 & \les L_{0}^{\frac{1}{2}}M^{2\de_{0}},
\end{align*}
where the spatial cut-off functions $\rho_{\le R}(x)$ and $\rho_{>R}(x)$ are introduced. From this and \eqref{eq-lzero},
we have
\begin{align*}
 & \abs{\mathcal I_{(\thez,\thet)}(s,\xi)-\wt{\mathcal I_{(\thez,\thet)}}(s,\xi)}\\
 & \les\int_{\R^{3+3}}\rho_{\le L_{0}}(\eta)|\eta|^{-2}\\
 & \qquad\quad\times\abs{\wh{f_{\thez}}(\xi+\eta)\bigbra{\wh{f_{\thet}}(\xi+\eta+\sigma),\wh{f_{\thet}}(\xi+\sigma)}-\wh{f_{\thez}}(\xi)\left|\wh{f_{\thet}}(\xi+\sigma)\right|^{2}}d\eta d\sigma\\
 & \les L_{0}^{\frac{3}{2}}M^{2\de_{0}}\ve_{1}^{3} \les M^{(-\frac{9}{8}+\frac{1}{20}+2\de_{0})}\ve_{1}^{3}.
\end{align*}
To conclude the proof, it remains to show 
\begin{align*}
\left|\sum_{\thet\in \{\pm \}}\wt{\mathcal I_{(\thez,\thet)}}(s,\xi)-i\partial_{s}B(s,\thez\xi)\wh{f_{\thez}}(s,\xi)\right|\les M^{-(1+\de_{0})}\bra{N_0}^{-10}\ve_{1}^{3}. 
\end{align*}
Setting 
$$\zeta:=\left(\thez\frac{\xi}{\bra{\xi}}-\thet\frac{\sigma}{\bra{\sigma}}\right),$$
we estimate 
\begin{align}\begin{aligned}\label{es}
 & \left|\sum_{\thet \in \{\pm\}}\wt{\mathcal I_{(\thez,\thet)}}(s,\xi)-i\partial_{s}B(s,\thez\xi)\wh{f_{\thez}}(s,\xi)\right|\\
 & \les\abs{\wh{f_{\thez}}(\xi)}\left|\iint_{\R^{3+3}}e^{is\eta\cdot\zeta}\rho_{\le L_{0}}(\eta)|\eta|^{-2}\abs{\wh{f_{\thet}}(\sigma)}^{2}d\eta d\sigma-\frac{(2\pi)^{3}}{s}\int_{\R^{3}}|\zeta|^{-1}\abs{\wh{f_{\thet}}(\sigma)}^{2}d\sigma\right|\\
 & \les\lim_{A\to\infty}\abs{\wh{f_{\thez}}(\xi)}\left|\int_{\R^{3}}\left|\int_{\R^{3}}e^{is\eta\cdot\zeta}|\eta|^{-2}\left(\rho_{\le L_{0}}(\eta)-\rho_{\le A}(\eta)\right)d\eta\right|\abs{\wh{f_{\thet}}(\sigma)}^{2}d\sigma\right|\\
 & \les\bra{N}^{-10}M^{(-\frac{5}{4}-2\de_{0})}\ve_{1}\left|\int_{\R^{3}}|\zeta|^{-2}\abs{\wh{f_{\thet}}(\sigma)}^{2}d\sigma\right|.
\end{aligned}\end{align}
Here we used the formula 
\[
\frac{(2\pi)^{3}}{s|\zeta|}=\lim_{A\to\infty}\int_{\R^3} e^{is\zeta\cdot\eta}\frac{1}{|\eta|^{2}}\rho_{\le A}(\eta)d\eta,
\]
which gives for $L_{0}\ll A$,
\begin{align*}
 & \left|\int_{\R^{3}}e^{is\eta\cdot\zeta}|\eta|^{-2}\left(\rho_{\le L_{0}}(\eta)-\rho_{\le A}(\eta)\right)d\eta\right|\\
 & =|s\zeta|^{-2}\left|\int_{\R^{3}}\left(\nabla_{\eta}^{2}e^{is\eta\zeta}\right)|\eta|^{-2}\left(\rho_{\le L_{0}}(\eta)-\rho_{\le A}(\eta)\right)d\eta\right|\\
 & \les M^{-2}L_{0}^{-1}|\zeta|^{-2}.
\end{align*}
Since $|\zeta|\gtrsim\min\left\{ 1,|\sigma|,\frac{|\thez\xi- \thet\sigma|}{\bra{\sigma}^{3}}\right\} $, a priori assumption \eqref{assumption-apriori} yields that
\[
\left|\int_{\R^{3}}|\zeta|^{-2}\abs{\wh{f_{\thet}}(\sigma)}^{2}d\sigma\right|\les\ve_{1}^{2}.
\]
Plugging this bound into \eqref{es}, we complete the proof of \eqref{eq-modi-part2}.

\smallskip

\textbf{Proof of \eqref{eq-modi-part2-2}.} 
When we estimate the low frequency parts $|\xi|\le s^{\frac3k}$ with $\mathbf{\Theta}\neq\mathbf{\Xi}$, where $\theta_0=-\theta_1$ or $\theta_2=-\theta_3$, the cancellation effect from the oscillatory term $e^{iB(s,\theta_{0}\xi)}$ could be ignored.
Instead, the spinorial structure from $\Pi_{\theta_0}(D)\Pi_{\theta_1}(D)$ or $\Pi_{\theta_2}(D)\Pi_{\theta_3}(D)$ plays an essential role to bound the singularity.
For instance, if $\theta_0=-\theta_1$, we estimate using \eqref{decay-null} 
\begin{align*}
	& \left|\mathcal I_{{\mathbf{\Theta}},L_0}(s,\xi)\right|\\
	& \les\int_{\R^{3+3}}\Big|\Pi_{\thez}(\xi)\Pi_{-\thez}(\xi+\eta)\Big| \left|\rho_{\le L_{0}}(\eta) \right| |\eta|^{-2}\\
	& \qquad\qquad\qquad\times\left|\wh{f_{-\thez}}(\xi+\eta)\bigbra{\wh{f_{\theta_3}}(\xi+\eta+\sigma),\wh{f_{\thet}}(\xi+\sigma)}\right|d\eta d\sigma\\
	& \les L_{0}^{2}\normo{f_{\theta_3}}_{L_{x}^{2}}^{2}\normo{f_{\thet}}_{L_{x}^{2}}^{2}\normo{\wh{f_{-\thez}}}_{L_{\xi}^{\infty}}\\
	& \les M^{\left(-\frac{3}{2}+\frac{1}{15}+2\de_{0}\right)}\ve_{1}^{3}.
   \end{align*}
The other cases can be estimated similarly.



\smallskip

\textbf{Proof of \eqref{eq-modi-part3}.} We localize the frequencies of $f_{\theta_i}$ as follows:
\begin{align*}
\mathcal I_{{\mathbf{\Theta}},L}(s,\xi) & =ic_{1}\sum_{\textbf{N}\in (2^{\Z})^3}\mathcal I_{{\mathbf{\Theta}},L}^{\mathbf{N}}(s,\xi),\\
\mathcal I_{{\mathbf{\Theta}},L}^{\mathbf{N}}(s,\xi) & =\iint_{\R^{3+3}}e^{is\phasep\freq}\Pi_{\thez}(\xi)|\eta|^{-2}\wh{f_{\theo,N_{1}}}(s,\xi+\eta)\\
 & \qquad\qquad\qquad\qquad\qquad\times\bigbra{\wh{f_{\theth,N_{3}}}(s,\xi+\eta+\sigma),\wh{f_{\thet,N_{2}}}(s,\xi+\sigma)}d\eta d\sigma,
\end{align*}
where $f_{\thej,N_{j}}=P_{N_{j}}f_{\thej}$, $\textbf{N}=(N_{1},N_{2},N_{3})$
is 3-tuple of dyadic numbers, and
\[
{p}_{{\mathbf{\Theta}} }(\xi,\eta,\sigma) = \thez\braxi-\theo\langle\xi+\eta\rangle-\thet\langle\xi+\sigma\rangle+\theth\bra{\xi+\eta+
	\sigma}.
\]
Then we have only to prove
\begin{align}
\sum_{L>L_{0},\textbf{N}\in(2^{\Z})^3} \left| \mathcal I_{{\mathbf{\Theta}},L}^{\mathbf{N}}(s,\xi)\right|  \les\ve_{1}^{3}M^{-1-\de_{0}}\bra{N}^{-10}.\label{eq-goal-high}
\end{align}
By H\"older inequality, we readily have
\[
\left|\mathcal I_{{\mathbf{\Theta}},L}^{\mathbf{N}}(s,\xi)\right|\les L^{-\frac{1}{2}}\prod_{j=1}^{3}\normo{\wh{f_{\thej,N_{j}}}(s)}_{L^{2}},
\]
and the a priori assumption implies the following bounds 
\[
\normo{\wh{f_{\thej,N_{j}}}(s)}_{L^{2}}\les\min{\left(N_{j}^{\frac{3}{2}}\bra{N_{j}}^{-10},\bra{N_{j}}^{-k}M^{\de_{0}}\right)}\ve_{1}\;\;\mbox{ for }\;\;j=1,2,3.
\]
The last two estimates above are sufficient to show \eqref{eq-goal-high} for the sum over those indices $\mathbf{N}$ satisfying $N_{\max}\ge M^{\frac3k}$
or $N_{\min}\le M^{-1}$. Here we denoted $\max{(N_{1},N_{2},N_{3})}$
and $\min{(N_{1},N_{2},N_{3})}$ by $N_{\max}$ and $N_{\min}$ ,
respectively. Then, it remains to bound the sum in  \eqref{eq-goal-high} over $L>L_0$ and $\mathbf{N}$ satisfying
\begin{align*}
M^{-1}\le N_{1},N_{2},N_{3}\le M^{\frac3k}.
\end{align*}
 Let us further localize $\sigma$ variable by
$L'\in2^{\Z}$ to write
\begin{align*}
\mathcal I_{{\mathbf{\Theta}},L}^{\mathbf{N}}(s,\xi)
&=\sum_{L' \in 2^{\Z}}\mathcal I_{{\mathbf{\Theta}},\textbf{L}}^{\mathbf{N}}(s,\xi), \\ 	
\mathcal I_{{\mathbf{\Theta}},\textbf{L}}^{\mathbf{N}}(s,\xi) & =\iint_{\R^{3+3}}e^{is\phasep\freq}\Pi_{\thez}(\xi)|\eta|^{-2}\rho_{L}(\eta)\rho_{L'}(\sigma)\wh{f_{\theo,N_{1}}}(s,\xi+\eta)\\
 & \qquad\qquad\qquad\qquad\qquad\times\bigbra{\wh{f_{\theth,N_{3}}}(s,\xi+\eta+\sigma),\wh{f_{\thet,N_{2}}}(s,\xi+\sigma)}d\eta d\sigma,
\end{align*}
where $\textbf{L}=(L,L')$. 
Then we suffice to prove that 
\begin{align*}
	\sum_{(\textbf{N},\textbf{L}) \in \mathcal{A}} 
\left| \mathcal I_{{\mathbf{\Theta}},\mathbf{L}}^{\mathbf{N}}(s,\xi) \right| \les\ve_{1}^{3}M^{-1-\de_{0}}\bra{N}^{-10},
\end{align*}
where the summation runs over 
\begin{align}
\mathcal{A}=\left\{ (\textbf{N},\textbf{L})\in (2^{\Z})^3 \times (2^{\Z})^2 :M^{-1}\le N_{1},N_{2},N_{3}\le M^{\frac3k},\;L_{0}\le L \right\} .\label{eq-condi-piece}
\end{align}
We observe that the other dyadic numbers also satisfy $N,L,L'\les M^{\frac3k}$.
We divide the summation over $\mathcal{A}$ into 
\begin{align*}
\sum_{(\textbf{N},\textbf{L}) \in \mathcal{A}}
=\sum_{(\textbf{N},\textbf{L}) \in \mathcal{A}_1}
+\sum_{(\textbf{N},\textbf{L}) \in \mathcal{A}_2}
+\sum_{(\textbf{N},\textbf{L}) \in \mathcal{A}_3}
+\sum_{(\textbf{N},\textbf{L}) \in \mathcal{A}_4},
\end{align*}
where 
\begin{align*}
\textbf{Case~1: } \mathcal{A}_1 &= \mathcal{A} \cap \{  L'\le L^{-\frac{1}{3}}M^{-\frac13 - \frac1k}  \},  \\ 
\textbf{Case~2: }\mathcal{A}_2 &= \mathcal{A} \cap \{ L'\ge L^{-\frac{1}{3}}M^{-\frac13 - \frac1k} \;\; \mbox{and}\;\; \max{(N_{1},N_{3})}\les L \}, \\ 
\textbf{Case~3: }\mathcal{A}_3 &= \mathcal{A} \cap \{ L'\ge L^{-\frac{1}{3}}M^{-\frac13 - \frac1k},\;\max(N_{1},N_{3})\gg L, \;\;\mbox{and}\;\; N_{1}\nsim N_{3} \}, \\ 
\textbf{Case~4: }\mathcal{A}_4 &= \mathcal{A} \cap \{ L'\ge L^{-\frac{1}{3}}M^{-\frac13 - \frac1k},\;\max(N_{1},N_{3})\gg L, \;\;\mbox{and}\;\; N_{1}\sim N_{3} \}.
\end{align*}


\global\long\def\nl{{\textbf{N},\textbf{L}}}%

\noindent In the estimates for the first three cases, the signs $\mathbf{\Theta}$ will play no role.

\smallskip

\noindent \textbf{Estimates for Case~1.} The low frequency part with respect to $\sigma$ variable can be estimated by the trivial estimates
\begin{align*}
\sum_{(\nl)\in\mathcal{A}_1}|\mathcal I_{{\mathbf{\Theta}},\textbf{L}}^{\mathbf{N}}(s,\xi)| & \les\sum_{(\nl)\in\mathcal{A}_1}L(L')^{3}\bra{N_{\max}}^{-10}\prod_{j=1}^{3}\normo{\bra{\xi}^{10}\wh{f_{\thej,N_{j}}}(s,\xi)}_{L_{\xi}^{\infty}}\\
 & \les\ve_{1}^{3}M^{-(1+\de_{0})}\bra{N}^{-10}.
\end{align*}

\smallskip

\noindent \textbf{Estimates for Case~2.} 
We apply Lemma \ref{kernel} to $\mathcal I_{{\mathbf{\Theta}},{\bf L}}^{{\bf N}}$ with
\[
m_{\nl}(\xi,\eta,\sigma):=|\eta|^{-2}\rho_{L}(\eta)\rho_{L'}(\sigma)\rho_{N_{1}}(\xi+\eta)\rho_{N_{3}}(\xi+\eta+\sigma)
\]
and $\|m_{\nl}(\xi)\|_{\text{CM}} \les L^{-2}$ to estimate
\begin{align*}
\sum_{(\nl)\in\mathcal{A}_2}|\mathcal I_{{\mathbf{\Theta}},\textbf{L}}^{\mathbf{N}}(s,\xi)| & \les\sum_{(\nl)\in\mathcal{A}_2}L^{-2}\|\wh{f_{\theo,N_{1}}}\|_{L^{2}}\|\psi_{\thet,N_2}\|_{L^{\infty}}\|\wh{f_{\theth,N_{3}}}\|_{L^{2}}\\
 & \les\ve_{1}^{3}M^{-\frac{3}{2}}\sum_{(\nl)\in\mathcal{A}_2}L^{-2}N_{1}^{\frac32}\bra{N_{1}}^{-10}\bra{N_2}^{-2}N_{3}^{\frac32}\bra{N_{3}}^{-10}\\
 & \les\ve_{1}^{3}M^{-\frac{3}{2}+\delta_{0}}\bra{N}^{-10}.
\end{align*}

\smallskip
\noindent \textbf{Estimates for Case~3.}
In this case we extract an extra time decay via the normal form approach. To achieve this, using the relation
\[
e^{is\phase}=-i\frac{1}{s}\frac{(\nabla_{\eta}\phase)\cdot\nabla_{\eta}e^{is\phase}}{|\nabla_{\eta}\phase|^{2}},
\]
we perform an integration by parts in $\eta$ to obtain  
\begin{align*}
\mathcal I_{{\mathbf{\Theta}},\textbf{L}}^{\mathbf{N}}(s,\xi) 
&=\mathcal J_{1}(s,\xi) 
+\mathcal J_{2}(s,\xi) ,\\
\mathcal J_{1}(s,\xi) & :=\frac{i}{s}\Pi_{\theta_0}(\xi)\iint_{\R^{3+3}}e^{is\phasep\freq}\textbf{m}_{1}(\xi,\eta,\sigma) \\ 
&\times \nabla_{\eta}\left( \wh{f_{\theo,N_{1}}}(s,\xi+\eta) \bigbra{\wh{f_{\theth,N_{3}}}(s,\xi+\eta+\sigma),\wh{f_{\thet,N_{2}}}(s,\xi+\sigma)} \right) d\eta d\sigma,\\
\mathcal J_{2}(s,\xi) & :=\frac{i}{s}\Pi_{\theta_0}(\xi)\iint_{\R^{3+3}}e^{is\phasep\freq}\textbf{m}_{2}(\xi,\eta,\sigma)\\
 &\times\wh{f_{\theo,N_{1}}}(s,\xi+\eta)\bigbra{\wh{f_{\theth,N_{3}}}(s,\xi+\eta+\sigma),\wh{f_{\thet,N_{2}}}(s,\xi+\sigma)}d\eta d\sigma,
\end{align*}
where
\begin{align}\label{eq:multi-m1}
	\begin{aligned} 
\textbf{m}_{1}(\xi,\eta,\sigma) & =\frac{\nabla_{\eta}\phase\freq}{|\nabla_{\eta}\phase\freq|^{2}}|\eta|^{-2}\rho_{L}(\eta)\rho_{L'}(\sigma)\rho_{N_{1}}(\xi+\eta)\rho_{N_{3}}(\xi+\eta+\sigma),\\
\mathbf{m}_{2}(\xi,\eta,\sigma) & =\nabla_{\eta}\left(\frac{\nabla_{\eta}\phase\freq}{|\nabla_{\eta}\phase\freq|^{2}}|\eta|^{-2}\rho_{L}(\eta)\rho_{L'}(\sigma)\rho_{N_{1}}(\xi+\eta)\rho_{N_{3}}(\xi+\eta+\sigma)\right).
\end{aligned}
\end{align}
We first claim that for fixed $\xi\in \R^3$, 
\begin{align}\begin{aligned}\label{cm m1}
	\| \mathbf{m}_1 \|_{\text{CM}}&:= \left\Vert \iint_{\mathbb{R}^{3+3}}\mathbf{m}_1(\xi,\eta,\sigma)e^{ix\cdot\eta}e^{iy\cdot\sigma} d\eta d\sigma\right\Vert _{L_{x,y}^{1}(\R^{3}\times\R^3)}  \\ 
	&\les \max(N_1,N_3)^{-1} \min(\langle N_1\rangle,\langle N_3\rangle)^{15}L^{-2}.
\end{aligned}\end{align}
Indeed, by changing variables, we have 
\begin{align*}
\| \mathbf{m}_1 \|_{\text{CM}}
=\| \widetilde{ \mathbf{m}_1 }  \|_{\text{CM}}, \end{align*}
where \begin{align*}
	\widetilde{ \mathbf{m}_1 }(\xi,\eta,\sigma)
&= \frac{q_{(\theta_1,\theta_3)}(\eta,\sigma)}{|q_{(\theta_1,\theta_3)}(\eta,\sigma)|^2}|\xi-\eta|^{-2}\rho_{L}(\xi-\eta)\rho_{L'}(\sigma-\eta)\rho_{N_1}(\eta)\rho_{N_3}(\sigma), \\ 
q_{(\theta_1,\theta_3)}(\eta,\sigma)&=\theta_1\frac{\eta}{\langle \eta\rangle}-\theta_3\frac{\sigma}{\langle \sigma\rangle}.
\end{align*}
Using the inequalities
\begin{align}\label{phase lowbound}
\left| \frac{\eta}{\langle \eta\rangle}\pm \frac{\sigma}{\langle \sigma\rangle} \right|  
\gtrsim \frac{ \left||\eta|-|\sigma|\right|}{\min(\langle \eta\rangle,\langle \sigma\rangle)\max(\langle \eta\rangle,\langle \sigma\rangle)^2 }, 
\end{align}
one can show that $\widetilde{ \mathbf{m}_1 }$ satisfies the following differential inequalities 
\begin{align*}
|\nabla_\eta^m\nabla_\sigma^n &\widetilde{ \mathbf{m}_1 }(\xi,\eta,\sigma)| \les C(\mathbf{N},\mathbf{L})L^{-m}N_3^{-n}\rho_{L}(\eta-\xi)\rho_{N_3}(\sigma),\\  
C(\mathbf{N},\mathbf{L})=&\begin{cases}
L^{-2}\min(N_1,N_3)^{2+2(m+n)} & \text{ if } \min(N_1,N_3) \ge 1, \\ 
L^{-2}\max(N_1,N_3)^{-1} \max( \langle N_1\rangle, \langle N_3\rangle)  & \text{ if } \min(N_1,N_3) \le 1,
\end{cases}
\end{align*}
which implies \eqref{cm m1} by a straightforward computation.
Hence, applying Lemma \ref{kernel} with \eqref{cm m1} to $\mathcal{J}_1$, we estimate 
\begin{align*}
 & \sum_{(\nl)\in\mathcal{A}_3}  |\mathcal J_{1}(s,\xi)| \\ 
&\les M^{-1+\frac{30}{k}}\bra{N}^{-10} \sum_{(\nl)\in\mathcal{A}_3}   \max(N_1,N_3)^{-1} \min(\langle N_1\rangle,\langle N_3\rangle)^{15}L^{-2} \\ 
 &\qquad \times 
 \left( \normo{\nabla \wh{f_{\theo,N_{1}}}}_{L^{2}} \normo{\wh{f_{\theth,N_{3}}}}_{L^{2}} 
 +
 \normo{\wh{f_{\theo,N_{1}}}}_{L^{2}}\normo{\nabla  \wh{ f_{\theth,N_{3}}}}_{L^{2}} 
 \right) \|\psi_{\thet,N_2}\|_{L^{\infty}} \\ 
 &  \les\ve_{1}^{3}M^{-1+\frac{30}k} \bra{N}^{-10} M^{\frac{3}{2}-\frac{1}{15}}M^{\de_{0}}M^{-\frac{3}{2}}M^{\frac{12}k} \les\ve_{1}^{3}M^{-(1+\de_{0})}\bra{N}^{-10},
\end{align*}
where we used $M^{-\frac{3}{4}+\frac{1}{30}}\sim L_{0}\le L$.
Next, in oder to estimate $\mathcal J_{2}(s,\xi)$ we perform an integration by parts
in $\eta$ once again to obtain
\begin{align*}
\mathcal J_{2}(s,\xi) & =\mathcal J_{3}(s,\xi)+\mathcal J_{4}(s,\xi),\\
\mathcal J_{3}(s,\xi) & :=-\frac{1}{s^{2}}\Pi_{\theta_0}(\xi)
\iint_{\R^{3+3}}e^{is\phasep\freq}
\textbf{m}_{3}(\xi,\eta,\sigma) \\ 
&\qquad \times\nabla_{\eta}\left( 
\wh{f_{\theo,N_{1}}}(s,\xi+\eta) \bigbra{\wh{f_{\theth,N_{3}}}(s,\xi+\eta+\sigma),\wh{f_{\thet,N_{2}}}(s,\xi+\sigma)}\right)d\eta d\sigma, \\
\mathcal J_{4}(s,\xi) & :=-\frac{1}{s^{2}}\Pi_{\theta_0}(\xi)\iint_{\R^{3+3}}e^{is\phasep\freq}\textbf{m}_{4}(\xi,\eta,\sigma) \\ 
&\qquad\times\wh{f_{\theo,N_{1}}}(s,\xi+\eta)
\bigbra{\wh{f_{\theth,N_{3}}}(s,\xi+\eta+\sigma),\wh{f_{\thet,N_{2}}}(s,\xi+\sigma)}d\eta d\sigma,
\end{align*}
where
\begin{align}
	\begin{aligned}\label{eq:multi-m21}
\textbf{m}_{3}(\xi,\eta,\sigma) & =\frac{\nabla_{\eta}\phase\freq}{|\nabla_{\eta}\phase\freq|^{2}}\textbf{m}_{2}\freq,\\
\textbf{m}_{4}(\xi,\eta,\sigma) & =\nabla_{\eta}\left(\frac{\nabla_{\eta}\phase\freq}{|\nabla_{\eta}\phase\freq|^{2}}\textbf{m}_{2}\freq\right).
\end{aligned}
\end{align}
An analogous computation as in the proof of \eqref{cm m1} yields that 
\begin{align*}
\| \mathbf{m}_3 \|_{\text{CM}}
\les \max(N_1,N_3)^{-2} \min(\langle N_1\rangle,\langle N_3\rangle)^{20}L^{-3}.
\end{align*}
Applying Lemma \ref{kernel} with the bound, we estimate
\begin{align*}
	& \sum_{(\nl)\in\mathcal{A}_3}|\mathcal J_{3}(s,\xi)|\\
	& \les M^{-2+\frac{30}k}\bra{N}^{-10}\sum_{(\nl)\in\mathcal{A}_3}L^{-3}\max({N_{1},N_{3}})^{-2}\min(\langle N_1\rangle,\langle N_3\rangle)^{20} \\ 
&\qquad \left( \normo{\nabla\wh{f_{\theo,N_{1}}}}_{L^{2}}\normo{\wh{f_{\theth,N_{3}}}}_{L^{2}} + \normo{\wh{f_{\theo,N_{1}}}}_{L^{2}}\normo{\nabla\wh{f_{\theth,N_{3}}}}_{L^{2}}  \right)\normo{\psi_{\thet,N_{2}}}_{L_{x}^{\infty}} \\ 
	&  \les\ve_{1}^{3}M^{-2+\frac{30}k} \bra{N}^{-10}M^{\de_{0}}M^{-\frac{3}{2}} M^{\frac{24}{k}}
\sum_{ L \ge M^{-\frac34 + \frac1{30}} } 
	L^{-3}
	\sum_{L' \ge L^{-\frac13}M^{-\frac13-\frac1k}}	(L')^{-\frac{1}{2}}\\
	& \les\ve_{1}^{3}M^{-\frac{7}{2}+ \frac{54}{k} +\de_0}M^{\frac{17}8 -\frac{17}{180}}M^{\frac{1}{6}+\frac1{2k}} \bra{N}^{-10}\les\ve_{1}^{3}M^{-(1+\de_{0})}\bra{N}^{-10}. 
\end{align*}
Next, from \eqref{phase lowbound}, one can find the pointiwse bound of $\mathbf{m}_4$
\begin{align*}
|\textbf{m}_{4}(\xi,\eta,\sigma)|\les \max(N_1,N_3)^{-2} \min(\langle N_1\rangle,\langle N_3\rangle)^{10}L^{-4}.
\end{align*}
Using this and H\"older inequality, we estimate 
\begin{align*}
 & \sum_{(\nl)\in\mathcal{A}_3}|\mathcal J_{4}(s,\xi)|\\
 & \les M^{-2+\frac{30}{k}}\bra{N}^{-10}\sum_{(\nl)\in\mathcal{A}_3}\max(N_1,N_3)^{-2} \min(\langle N_1\rangle,\langle N_3\rangle)^{10}L^{-1} \\ &\qquad\qquad\qquad\times\normo{\wh{f_{\theo,N_{1}}}}_{L^{2}}\normo{\wh{f_{\thet,N_{2}}}}_{L^{\infty}}\normo{\wh{f_{\theth,N_{3}}}}_{L^{2}}\\
 & \les\ve_{1}^{3}M^{-2+\frac{30}{k}+\de_0}\bra{N}^{-10}
 \sum_{L \ge M^{-\frac34 + \frac1{30}}} L^{-1}
 \sum_{L' \ge M^{-\frac25}} (L')^{-\frac{1}{2}}
 \les\ve_{1}^{3}M^{-(1+\de_{0})}\bra{N}^{-10},
\end{align*}
where we used
\begin{align*}
M^{-\frac{2}{5}}\les L',\qquad L'\sim|\sigma|\les\max(N_{1},N_{3}),
\end{align*}
which easily follows from the condition on dyadic pieces \eqref{eq-condi-piece}.

\smallskip
\noindent \textbf{Estimates for Case~4.}
As in \textbf{Case~3}, we perform an integration by parts to gain a time decay. The strategy depends on combinations of sign $\mathbf{\Theta}$. 
\begin{enumerate}
\item If $\theta_1=\theta_3$, we perform an integration by parts with respect to $\eta$ variable using the following 
\begin{align*}
	|\nabla_{\eta}p_{\mathbf\Theta}(\xi,\eta,\sigma)| = \left| \theta_1\left( \frac{\xi+\eta}{\langle \xi+\eta\rangle} - \frac{\xi+\eta+\sigma}{\langle \xi+\eta+\sigma\rangle} \right) \right| 
	\gtrsim L' \langle N_1\rangle^{-3}.
\end{align*}
\item If $\theta_1\neq\theta_3$, we perform an integration by parts with respect to $\sigma$ variable using the following 
\begin{align}\label{psigma}
	|\nabla_{\sigma}p_{\mathbf\Theta}(\xi,\eta,\sigma)| \gtrsim   
\left\{ \begin{aligned}
	&\left|\frac{\xi+\sigma}{\langle \xi+\sigma\rangle}- \frac{\xi+\eta+\sigma}{\langle \xi+\eta+\sigma\rangle} \right| \gtrsim L\langle N_2\rangle^{-3}, &&\text{ if } \theta_2=\theta_3, \\ 
	&\left|\frac{\xi+\sigma}{\langle \xi+\sigma\rangle} + \frac{\xi+\eta+\sigma}{\langle \xi+\eta+\sigma\rangle} \right| \gtrsim N_2\langle N_2\rangle^{-1}, &&\text{ if } \theta_2 \neq \theta_3. \\ 
\end{aligned}\right. 
   \end{align}
\end{enumerate}
Here we make an observation for \eqref{psigma} when $\thet \neq \theth$. If $L \sim N_2 \sim N_3$, $|\nabla_{\sigma} p_{\mathbf{\Theta}}|$ does not have lower bound. However, in this case, we only consider the case $L \ll N_3$ which makes the lower bound as in \eqref{psigma} when $\thet \neq \theth$.

\noindent \textit{(1) Estimates when $\theta_1=\theta_3$.} We perform an integration
by parts twice in $\eta$ variable  to obtain 
\begin{align*}
\mathcal I_{{\mathbf{\Theta}},\textbf{L}}^{\textbf{N}}(s,\xi) & =\mathcal K_{1}(s,\xi)+\mathcal K_{2}(s,\xi)+\mathcal K_{3}(s,\xi),\\
\mathcal K_{1}(s,\xi) & :=-\frac{1}{s^{2}}\iint_{\R^{3+3}}e^{is\phasep\freq}\textbf{k}_{1}(\xi,\eta,\sigma)\nabla_{\eta}^{2}\left[\wh{f_{\theo,N_{1}}}(s,\xi+\eta)\right.\\
 & \qquad\qquad\qquad\qquad\qquad\left.\times\bigbra{\wh{f_{\theth,N_{3}}}(s,\xi+\eta+\sigma),\wh{f_{\thet,N_{2}}}(s,\xi+\sigma)}\right]d\eta d\sigma,\\
\mathcal K_{2}(s,\xi) & :=-\frac{2}{s^{2}}\iint_{\R^{3+3}}e^{is\phasep\freq}\textbf{k}_{2}(\xi,\eta,\sigma)\nabla_{\eta}\left[\wh{f_{\theo,N_{1}}}(s,\xi+\eta)\right.\\
 & \qquad\qquad\qquad\qquad\qquad\left.\times\bigbra{\wh{f_{\theth,N_{3}}}(s,\xi+\eta+\sigma),\wh{f_{\thet,N_{2}}}(s,\xi+\sigma)}\right]d\eta d\sigma,\\
\mathcal K_{3}(s,\xi) & :=-\frac{1}{s^{2}}\iint_{\R^{3+3}}e^{is\phasep\freq}\textbf{k}_{3}(\xi,\eta,\sigma)\wh{f_{\theo,N_{1}}}(s,\xi+\eta)\\
 & \qquad\qquad\qquad\qquad\qquad\times\bigbra{\wh{f_{\theth,N_{3}}}(s,\xi+\eta+\sigma),\wh{f_{\thet,N_{2}}}(s,\xi+\sigma)}d\eta d\sigma,
\end{align*}
where
\begin{align*}
\textbf{k}_{1}(\xi,\eta,\sigma) & :=\frac{\nabla_{\eta}\phase\freq}{|\nabla_{\eta}\phase\freq|^{2}}\textbf{m}_{1}(\xi,\eta,\sigma),\\
\textbf{k}_{2}(\xi,\eta,\sigma) & :=\frac{\nabla_{\eta}\phase\freq}{|\nabla_{\eta}\phase\freq|^{2}}\textbf{m}_{2}(\xi,\eta,\sigma),\\
\textbf{k}_{3}(\xi,\eta,\sigma) & :=\textbf{m}_{3}(\xi,\eta,\sigma).
\end{align*}
Here ${\bf m}_{1},{\bf m}_{2}$, and ${\bf m}_{3}$ are defined in \eqref{eq:multi-m1} and \eqref{eq:multi-m21}, respectively. 
Since $\theta_1=\theta_3$, by the mean value theorem we have 
\begin{equation*}
|{ \nabla_{\eta}p}_{{\mathbf{\Theta}}}(\xi,\eta,\sigma)|= \left|\frac{\xi+\eta}{\bra{\xi+\eta}} - \frac{\xi+\eta+\sigma}{\bra{\xi+\eta+\sigma}} \right|\gtrsim\frac{L'}{\bra{N_{1}}^{3}},
\end{equation*}
which leads us to the differential inequalities
\begin{align*}
\begin{aligned}
\left|\nabla_{\eta}^{m}\nabla_{\sigma}^{\ell}\textbf{k}_{1}\freq\right| & \les L^{-2}(L')^{-2}\bra{N_{1}}^{6} \bra{N_1}^{2(m+\ell)} L^{-m}(L')^{-\ell} \\
&\les L^{-2}(L')^{-2}\bra{N_{1}}^{6} M^{\frac{48}{k}} L^{-m}(L')^{-\ell},\\
\left|\nabla_{\eta}^{m}\nabla_{\sigma}^{\ell}\textbf{k}_{2}\freq\right| & \les L^{-3}(L')^{-2}\bra{N_{1}}^{6}\bra{N_1}^{2(m+\ell)}L^{-m}(L')^{-\ell}\\
&\les L^{-3}(L')^{-2}\bra{N_{1}}^{6}M^{\frac{48}k}L^{-m}(L')^{-\ell},
\end{aligned}
\end{align*}
for $0\le m,\ell \le 4$, and the pointwise bound
\begin{align}
\left|\textbf{k}_{3}\freq\right| & \les L^{-4}(L')^{-2}\bra{N_{1}}^{6}.\label{eq-muliple-k}
\end{align}
Applying Lemma~\ref{kernel} with $\| \mathbf{k_1} \|_{\rm{CM}}=L^{-2}(L')^{-2}\bra{N_{1}}^{6} M^{\frac{48}{k}}$, we estimate 
\begin{align*}
 & \sum_{(\nl)\in \mathcal{A}_4 }\left|\mathcal K_{1}(s,\xi)\right|\\
 & \les M^{-2+\frac{48}k + \frac{30}k}  \bra{N}^{-10}\sum_{(\nl)\in \mathcal{A}_4 }L^{-2}(L')^{-2}\bra{N_{1}}^{6}\normo{\psi_{\thet,N_2}}_{L^{\infty}}\\
 &\hspace{1cm}\times \left(N_1^{-2}\normo{x^{2}f_{\theo}}_{L^{2}}\normo{f_{\theth,N_{3}}}_{L^{2}} + N_3^{-2} \normo{f_{\theo,N_1}}_{L^{2}}\normo{x^2f_{\theth}}_{L^{2}}\right)\\
 & \les 
 \ve_1^3 M^{-2 +\frac{78}k} \bra{N}^{-10} M^{-\frac32 +2\de_0}\sum_{(\nl)\in \mathcal{A}_4 }L^{-2}(L')^{-2}\bra{N_2}^{-2}N_1^{-\frac12}\bra{N_{1}}^{-4}\\
 & \les\ve_{1}^{3}M^{-\frac72+\frac{78}k+2\de_0}\bra{N}^{-10}M^{1-\frac{2}{45}}M^{\frac23 +\frac2k} M^{\frac38 -\frac1{60}}    
 \les\ve_{1}^{3}M^{-(1+\de_{0})}\bra{N}^{-10},
\end{align*}
where we used that 
$$
L^{-\frac{2}{3}}(L')^{-2}\le M^{\frac23 + \frac2k} \;\mbox{ and }\; M^{-\frac34 + \frac1{30}} \les L \ll N_1.
$$
Similarly, we estimate by using $\| \mathbf{k_2} \|_{\rm{CM}}=L^{-3}(L')^{-2}\bra{N_{1}}^{6} M^{\frac{48}{k}}$
\begin{align*}
 & \sum_{(\nl)\in\mathcal{A}_4}\left|\mathcal K_{2}(s,\xi)\right|\\
 & \les M^{-2+\frac{48}k +\frac{30}k}\bra{N}^{-10} \sum_{(\nl)\in\mathcal{A}_4}L^{-3}(L')^{-2}\bra{N_{1}}^{6}\normo{\psi_{\thet,N_2}}_{L^{\infty}}\\
 &\hspace{1cm}\times \left(N_1^{-1}\normo{xf_{\theo}}_{L^{2}}\normo{f_{\theth,N_{3}}}_{L^{2}} + N_3^{-1} \normo{f_{\theo,N_1}}_{L^{2}}\normo{xf_{\theth}}_{L^{2}}\right)\\
 & \les\ve_{1}^{3}M^{-2+\frac{78}k}\bra{N}^{-10}M^{\frac{7}{4}-\frac{7}{90}}M^{\frac23 + \frac2k}M^{-\frac{3}{2}}M^{\de_{0}}\les\ve_{1}^{3}M^{-(1+\de_{0})}\bra{N}^{-10},
\end{align*}
where we used that 
\[
L^{-\frac{7}{3}}\le M^{\frac{7}{4}-\frac{7}{90}}\;\mbox{ and }\;L^{-\frac{2}{3}}(L')^{-2}\le M^{\frac23 + \frac2k}.
\]
Using H\"older inequality and \eqref{eq-muliple-k}, we see that
\begin{align*}
 & \sum_{(\nl)\in\mathcal{A}_4}\left|\mathcal K_{3}(s,\xi)\right|\\
 & \les M^{-2} \sum_{(\nl)\in\mathcal{A}_4}L^{-1}L' \bra{N_{1}}^{6}\normo{\wh{f_{\theo,N_{1}}}}_{L^{\infty}}\normo{\wh{f_{\thet,N_{2}}}}_{L^{\infty}}\normo{\wh{f_{\theth,N_{3}}}}_{L^{\infty}}\\
 & \les\ve_{1}^{3}M^{-2}M^{\frac{3}{4}-\frac{1}{30}} \bra{N}^{-10}\les\ve_{1}^{3}M^{-(1+\de_{0})}\bra{N}^{-10}.
\end{align*}

\noindent \textit{(2) Estimates when $\theta_1\neq\theta_3$.}
The proof proceeds similarly to the previous case. We integrate by parts twice in $\sigma$ variable
\begin{align*}
	\mathcal I_{{\mathbf{\Theta}},\textbf{L}}^{\textbf{N}}(s,\xi) & =\wt{\mathcal K_{1}}(s,\xi)+ \wt{\mathcal  K_{2}}(s,\xi)+ \wt{\mathcal K_{3}}(s,\xi),\\
	\wt{\mathcal K_{1}}(s,\xi) & :=-\frac{1}{s^{2}}\iint_{\R^{3+3}}e^{is\phasep\freq}\wt{\textbf{k}_{1}}(\xi,\eta,\sigma)\nabla_{\sigma}^{2}\left[\wh{f_{\theo,N_{1}}}(s,\xi+\eta)\right.\\
	& \qquad\qquad\qquad\qquad\qquad\left.\times\bigbra{\wh{f_{\theth,N_{3}}}(s,\xi+\eta+\sigma),\wh{f_{\thet,N_{2}}}(s,\xi+\sigma)}\right]d\eta d\sigma,\\
	\wt{\mathcal K_{2}}(s,\xi) & :=-\frac{2}{s^{2}}\iint_{\R^{3+3}}e^{is\phasep\freq}\wt{\textbf{k}_{2}}(\xi,\eta,\sigma)\nabla_{\sigma}\left[\wh{f_{\theo,N_{1}}}(s,\xi+\eta)\right.\\
	& \qquad\qquad\qquad\qquad\qquad\left.\times\bigbra{\wh{f_{\theth,N_{3}}}(s,\xi+\eta+\sigma),\wh{f_{\thet,N_{2}}}(s,\xi+\sigma)}\right]d\eta d\sigma,\\
	\wt{\mathcal K_{3}}(s,\xi) & :=-\frac{1}{s^{2}}\iint_{\R^{3+3}}e^{is\phasep\freq}\wt{\textbf{k}_{3}}(\xi,\eta,\sigma)\wh{f_{\theo,N_{1}}}(s,\xi+\eta)\\
	& \qquad\qquad\qquad\qquad\qquad\times\bigbra{\wh{f_{\theth,N_{3}}}(s,\xi+\eta+\sigma),\wh{f_{\thet,N_{2}}}(s,\xi+\sigma)}d\eta d\sigma,
\end{align*}
where
\begin{align*}
	&\wt{\textbf{k}_{1}}(\xi,\eta,\sigma)  :=\left[\frac{\nabla_{\sigma}\phase\freq}{|\nabla_{\sigma}\phase\freq|^{2}}\right]^2|\eta|^{-2}\rho_{L}(\eta)\rho_{L'}(\sigma)\rho_{N_{2}}(\xi+\sigma)\rho_{N_{3}}(\xi+\eta+\sigma),\\
	&\wt{\textbf{k}_{2}}(\xi,\eta,\sigma):=\frac{\nabla_{\sigma}\phase\freq}{|\nabla_{\sigma}\phase\freq|^{2}}\nabla_{\sigma}\left(\frac{\nabla_{\sigma}\phase\freq}{|\nabla_{\sigma}\phase\freq|^{2}}\rho_{L'}(\sigma)\right)|\eta|^{-2}\\
	&\hspace{6cm}\times\rho_{L}(\eta)\rho_{N_{2}}(\xi+\sigma)\rho_{N_{3}}(\xi+\eta+\sigma),\\
	&\wt{\textbf{k}_{3}}(\xi,\eta,\sigma)  :=  \nabla_\sigma	\wt{\textbf{k}_{2}}(\xi,\eta,\sigma).
\end{align*}
As observed in \eqref{psigma}, $|\nabla_{\sigma}\phase|$ has different lower bounds depending on the sign of $\theta_2$ and $\theta_3$. 
One can verify that, however, if $(\mathbf{N},\mathbf{L})\in \mathcal{A}_4$, the lower bounds of $|\nabla_{\sigma}\phase|$ when $\theta_2\neq\theta_3$  are always greater than those when $\theta_2=\theta_3$
\begin{align*}
	N_2\langle N_2\rangle^{-1}\gtrsim L\langle N_2\rangle^{-3}.
\end{align*}
Hence, we suffice to treat the latter case, when $\theta_2=\theta_3$. 
Using \eqref{psigma}, we can show the pointwise bounds of the multipliers 
\begin{align*}
	\left|  \wt{\textbf{k}_1}\freq \right| &\les L^{-4} \bra{N_2}^{6},\\
	\left|  \wt{\textbf{k}_2}\freq \right| &\les L^{-4}(L')^{-1} M^{\frac6k}\bra{N_2}^{6},\\
	\left|  \wt{\textbf{k}_3}\freq \right| &\les L^{-4}(L')^{-2} M^{\frac6k}\bra{N_2}^{6}.
\end{align*}
Then, we estimate by the H\"older inequality
\begin{align*}
	& \sum_{(\nl)\in\mathcal{A}_4}\left|\wt{\mathcal K_{1}}(s,\xi)\right|\\
	& \les M^{-2} \sum_{(\nl)\in\mathcal{A}_4}L^{-4}\bra{N_{2}}^{6}L^3\normo{\wh{f_{\theo,N_{1}}}}_{L^{\infty}}\\
	&\hspace{3cm}\times \left(N_2^{-2}\normo{x^{2}f_{\thet}}_{L^{2}}\normo{f_{\theth,N_{3}}}_{L^{2}} + N_3^{-2} \normo{f_{\theo,N_2}}_{L^{2}}\normo{x^2f_{\theth}}_{L^{2}}\right)\\
	& \les \ve_1^3 M^{-2+2\de_0} \sum_{(\nl)\in\mathcal{A}_4}L^{-1}  (L')^{-\frac12}  \bra{N_3}^{-14}\\
	& \les\ve_{1}^{3}M^{-2+\de_0}M^{\frac{15}{24}-\frac{1}{36}}M^{\frac16 + \frac1{2k}}\bra{N}^{-10}\les\ve_{1}^{3}M^{-(1+\de_{0})}\bra{N}^{-10},
\end{align*}
and
\begin{align*}
	& \sum_{(\nl)\in\mathcal{A}_4}\left| \wt{\mathcal K_{2}}(s,\xi)\right|\\
	& \les M^{-2+\frac6k} \sum_{(\nl)\in\mathcal{A}_4}L^{-4}(L')^{-1}\bra{N_2}^{6}L^3\normo{\wh{f_{\theo,N_{1}}}}_{L^{\infty}}\\
	&\hspace{3cm}\times \left(N_2^{-1}\normo{xf_{\thet}}_{L^{2}}\normo{f_{\theth,N_{3}}}_{L^{2}} + N_3^{-1} \normo{f_{\theo,N_2}}_{L^{2}}\normo{xf_{\theth}}_{L^{2}}\right)\\
	& \les\ve_{1}^{3}M^{-2+\frac6k}M^{\frac{1}{2}-\frac{2}{90}}M^{\frac13 + \frac1k}M^{\de_{0}} \bra{N}^{-10}\les\ve_{1}^{3}M^{-(1+\de_{0})}\bra{N}^{-10}.
\end{align*}
where we used that
\[
L' \ll N_2 \sim N_3, \; L^{-1}\le M^{\frac34 - \frac1{30}},\mbox{ and }\;L^{-\frac{1}{3}}(L')^{-1}\le M^{\frac13 + \frac1k}.
\]
In a similar way, we get
\begin{align*}
	& \sum_{(\nl)\in\mathcal{A}_4}\left| \wt{\mathcal K_{3}}(s,\xi)\right|\\
	& \les M^{-2+\frac6k} \sum_{(\nl)\in\mathcal{A}_4}L^{-4}(L')^{-2}\bra{N_2}^{6}L^3(L')^3\normo{\wh{f_{\theo,N_{1}}}}_{L^{\infty}}\normo{\wh{f_{\thet,N_{2}}}}_{L^{\infty}}\normo{\wh{f_{\theth,N_{3}}}}_{L^{\infty}}\\
	& \les\ve_{1}^{3}M^{-2+\frac6k}M^{\frac{3}{4}-\frac{1}{30}} \bra{N}^{-10}\les\ve_{1}^{3}M^{-(1+\de_{0})}\bra{N}^{-10}.
\end{align*}
This concludes the proof of \eqref{eq-modi-part3}, and thus of Proposition \ref{prop-scatt}.


\appendix

\bibliographystyle{plain}
\bibliography{ReferencesCKLY}

\begin{thebibliography}{10}

\bibitem{Bejenaru2015}
Ioan Bejenaru and Sebastian Herr.
\newblock The cubic {D}irac equation: small initial data in {$H^1(\Bbb{R}^3)$}.
\newblock {\em Comm. Math. Phys.}, 335(1):43--82, 2015.

\bibitem{Bejenaru2017}
Ioan Bejenaru and Sebastian Herr.
\newblock On global well-posedness and scattering for the massive
  {D}irac-{K}lein-{G}ordon system.
\newblock {\em J. Eur. Math. Soc. (JEMS)}, 19(8):2445--2467, 2017.

\bibitem{bour}
Nikolaos Bournaveas.
\newblock Local existence for the {M}axwell-{D}irac equations in three space
  dimensions.
\newblock {\em Communications in Partial Differential Equations},
  21(5-6):693--720, 1996.

\bibitem{chagla}
J.~M. Chadam and R.~T. Glassey.
\newblock On the {M}axwell-{D}irac equations with zero magnetic field and their
  solution in two space dimensions.
\newblock {\em J. Math. Anal. Appl.}, 53(3):495--507, 1976.

\bibitem{chhole}
Yonggeun Cho, Seokchang Hong, and Kiyeon Lee.
\newblock Scattering and non-scattering of the {H}artree-type nonlinear {D}irac
  system at critical regularity.
\newblock {\em preprint}, 2021.

\bibitem{chhwya}
Yonggeun Cho, Gyeongha Hwang, and Changhun Yang.
\newblock On the modified scattering of {$3$}-d {H}artree type fractional
  {S}chr\"odinger equations with {C}oulomb potential.
\newblock {\em Adv. Differential Equations}, 23(9-10):649--692, 2018.

\bibitem{chleoz}
Yonggeun Cho, Kiyeon Lee, and Tohru Ozawa.
\newblock Small data scattering of 2d {H}artree type {D}irac equations.
\newblock {\em Journal of Mathematical Analysis and Applications},
  506(1):125549, 2022.

\bibitem{choz2006-siam}
Yonggeun Cho and Tohru Ozawa.
\newblock On the semirelativistic {H}artree--type equation.
\newblock {\em SIAM Journal on Mathematical Analysis}, 38(4):1060--1074, 2006.

\bibitem{anfosel}
Piero D'Ancona, Damiano Foschi, and Sigmund Selberg.
\newblock Null structure and almost optimal local regularity for the
  {D}irac-{K}lein-{G}ordon system.
\newblock {\em J. Eur. Math. Soc. (JEMS)}, 9(4):877--899, 2007.

\bibitem{anfosel2}
Piero D'Ancona, Damiano Foschi, and Sigmund Selberg.
\newblock Null structure and almost optimal local well-posedness of the
  {M}axwell-{D}irac system.
\newblock {\em American Journal of Mathematics}, 132(3):771--839, 2010.

\bibitem{ansel}
Piero D'Ancona and Sigmund Selberg.
\newblock Global well-posedness of the {M}axwell-{D}irac system in two space
  dimensions.
\newblock {\em Journal of Functional Analysis}, 260(8):2300--2365, 2011.

\bibitem{gaoh}
Cristian Gavrus and Sung jin Oh.
\newblock Global well-posedness of high dimensional {M}axwell-{D}irac for small
  critical data.
\newblock {\em Memoirs of the American Mathematical Society}, 264(1279), 2020.

\bibitem{gemasha2008}
Pierre Germain, Nader Masmoudi, and Jalal Shatah.
\newblock {Global Solutions for 3D Quadratic {S}chr\"odinger equations}.
\newblock {\em International Mathematics Research Notices}, 2009(3):414--432,
  12 2008.

\bibitem{gemasha2012-jmpa}
Pierre Germain, Nader Masmoudi, and Jalal Shatah.
\newblock {Global solutions for 2D quadratic {S}chr\"odinger equations}.
\newblock {\em Journal de Mathematiques Pures et Appliquees}, 97(5):505--543,
  2012.

\bibitem{gemasha2012-annals}
Pierre Germain, Nader Masmoudi, and Jalal Shatah.
\newblock Global solutions for the gravity water waves equation in dimension 3.
\newblock {\em Annals of Mathematics}, 175:691--754, 2012.

\bibitem{gepuro}
Pierre Germain, Fabio Pusateri, and Fr\'{e}d\'{e}ric Rousset.
\newblock Asymptotic stability of solitons for m{K}d{V}.
\newblock {\em Adv. Math.}, 299:272--330, 2016.

\bibitem{gro}
Leonard Gross.
\newblock The {C}auchy problem for the coupled {M}axwell and {D}irac equations.
\newblock {\em Communications on Pure and Applied Mathematics}, 19:1--15, 1966.

\bibitem{Hayashi-Naumkin}
Nakao Hayashi and Pavel~I. Naumkin.
\newblock Asymptotics for large time of solutions to the nonlinear
  {S}chr\"{o}dinger and {H}artree equations.
\newblock {\em Amer. J. Math.}, 120(2):369--389, 1998.

\bibitem{hele2014}
Sebastian Herr and Enno Lenzmann.
\newblock The {B}oson star equation with initial data of low regularity.
\newblock {\em Nonlinear Analysis: Theory, Methods \& Applications},
  97:125--137, 2014.

\bibitem{Herr2015}
Sebastian Herr and Achenef Tesfahun.
\newblock Small data scattering for semi-relativistic equations with {H}artree
  type nonlinearity.
\newblock {\em J. Differential Equations}, 259(10):5510--5532, 2015.

\bibitem{huh}
Hyungjin Huh.
\newblock Global charge solutions of {M}axwell-{D}irac equations in {$\Bbb
  R^{1+1}$}.
\newblock {\em Journal of Physics A: Mathematical and Theoretical},
  43(44):445206, 2010.

\bibitem{iopu}
Alexandru~D. Ionescu and Fabio Pusateri.
\newblock Nonlinear fractional {S}chr\"odinger equations in one dimension.
\newblock {\em Journal of Functional Analysis}, 266(1):139--176, 2014.

\bibitem{kapu}
Jun Kato and {Fabio} Pusateri.
\newblock A new proof of long-range scattering for critical nonlinear
  {S}chr{\"o}dinger equations.
\newblock {\em Differential and Integral Equations}, 24(9-10):923--940, 2011.

\bibitem{Ozawa91}
Tohru Ozawa.
\newblock Long range scattering for nonlinear {S}chr\"{o}dinger equations in
  one space dimension.
\newblock {\em Comm. Math. Phys.}, 139(3):479--493, 1991.

\bibitem{pusa}
Fabio Pusateri.
\newblock Modified scattering for the {B}oson star equation.
\newblock {\em Comm. Math. Phys.}, 332(3):1203--1234, 2014.

\end{thebibliography}

\end{document}